\documentclass{amsart}

\usepackage[margin=1in]{geometry}
\usepackage{url}
\usepackage{amsmath,amssymb}
\usepackage{amsthm}
\usepackage{stmaryrd}
\usepackage{siunitx}
\usepackage{commath}
\usepackage{graphicx}
\usepackage[caption=false]{subfig}
\usepackage[numbers,square]{natbib}
\usepackage{enumerate}
\usepackage{bm}
\usepackage{lipsum}
\usepackage{amsfonts}
\usepackage{epstopdf}
\usepackage{algorithmic}
\usepackage{amsopn}
\usepackage{mathtools}
\usepackage{multirow}
\usepackage{booktabs}
\usepackage{tikz}
\usepackage{enumitem}
\usepackage{eqnarray}
\usetikzlibrary{math}
\usepackage{hyperref}
\hypersetup{
		colorlinks   = true, %Colours links instead of ugly boxes
		urlcolor     = blue, %Colour for external hyperlinks
		linkcolor    = blue, %Colour of internal links
		citecolor   = red %Colour of citations
	}
\usepackage{cleveref}
\numberwithin{equation}{section}

\newtheorem{theorem}{Theorem}[section]
\newtheorem{lemma}[theorem]{Lemma}

\newtheorem{corollary}[theorem]{Corollary}

\newcommand{\bfn}{{\boldsymbol n}}

\theoremstyle{definition}

\theoremstyle{remark}
\newtheorem{remark}[theorem]{Remark}
\makeatletter
\newcommand{\tnorm}{\@ifstar\@tnorms\@tnorm}
\newcommand{\@tnorms}[1]{%
	\left|\mkern-1.5mu\left|\mkern-1.5mu\left|
	#1
	\right|\mkern-1.5mu\right|\mkern-1.5mu\right|
}
\newcommand{\@tnorm}[2][]{%
	\mathopen{#1|\mkern-1.5mu#1|\mkern-1.5mu#1|}
	#2
	\mathclose{#1|\mkern-1.5mu#1|\mkern-1.5mu#1|}
}
\makeatother

\newcommand{\vertiii}[2][1]{\abs[#1]{\kern0.15ex\norm[#1]{#2}\kern-0.25ex}}
\newcommand{\D}{\mathrm{D}}

\newcommand{\bfu}{\boldsymbol{u}}
\newcommand{\bfx}{\boldsymbol{x}}
\newcommand{\bfchi}{\boldsymbol{\chi}}

\usepackage{xspace,color}

\title[DG method for incompressible three-phase flow in porous media]{Existence and convergence of a discontinuous Galerkin method for the incompressible three-phase flow problem in porous media}
\author[G. Sosa Jones]{Giselle Sosa Jones}
\address{Department of Mathematics, University of Houston}
\email{ggsosajo@central.uh.edu}

\author[B. Riviere]{Beatrice Riviere}
\address{Department of Computational and Applied Mathematics, Rice University}
\email{riviere@rice.edu}

\author[L. Cappanera]{Lo\"{i}c Cappanera}
\address{Department of Mathematics, University of Houston}
\email{lmcappan@central.uh.edu}

\begin{document}
	
	\maketitle
	
	%---------------------------------------------
	\begin{abstract}
		{This paper presents and analyzes a discontinuous Galerkin method for the incompressible three-phase flow problem in porous media. We use a first order time extrapolation which allows us to solve the equations implicitly and sequentially. We show that the discrete problem is well-posed, and obtain a priori error estimates. Our numerical results validate the theoretical results, i.e. the algorithm converges with first order.}
		% Keywords:
		{discontinuous Galerkin; three-phase flow; porous media; a priori error estimates.}
	\end{abstract}
	
	\maketitle
	%---------------------------------------------
	
	%\keywords
	%---------------------------------------------
	
	%=============================================
	\section{Introduction}
	
	%=============================================
	\label{sec:introduction}
	
	Subsurface modeling is important in improving the efficiency of clean-up strategies of contaminated subsurface or
	the long-term storage of carbon dioxide in subsurface. Incompressible systems  of liquid phase, aqueous phase and vapor phase are
	mathematically modeled by nonlinear coupled partial differential equations that are challenging to analyze.
	This work formulates a numerical scheme for solving for the liquid pressure, the aqueous saturation and the vapor saturation
	using  discontinuous Galerkin methods in space
	and sequential implicit time stepping. This choice of primary unknown is
	inspired from previous work of  \cite{shank1989practical,hajibeygi2014compositional,Cappanera:2019}.
	Existence and uniqueness of the solutions is proved and convergence of the numerical method
	is obtained by deriving a priori error estimates.
	While the literature on computational modeling of three-phase flows is vast, to our knowledge, there are no papers on the theoretical analysis of the discretization of the three-phase flow problem.
	
	Ideal numerical methods for modeling multiphase flow in porous media are to 
	be locally mass conservative to accurately track the propagation
	of the phases through the media. Heterogeneities of the porous media include
	highly discontinuous permeability fields with possibly local
	geological features like pinch-out.  This implies that the numerical methods should handle discontinuous coefficients and unstructured grids. Discontinuous Galerkin methods are suitable methods thanks to their flexibility derived from the lack of continuity constraint between approximations on neighboring cells. DG are known to be locally mass conservative, to handle highly varying permeability fields and to be accurate and robust on unstructured meshes. For these reasons, the literature on DG methods for porous media flows has exponentially increased
	over the last twenty years.  The main drawback of these methods is their cost, which is higher than the cost of low order finite difference methods and finite volume methods.  DG has been applied to incompressible three-phase flow in \cite{DongRiviere2016} and to compressible three-phase flow in
	\cite{RankinRiviere2015,CappaneraRiviereSPE,CappaneraRiviere}. In the absence of capillary pressure, DG is combined with finite volume method in \cite{NatvigLie2008}, and with mixed finite element method in \cite{moortgat2013higher,moortgat2016mixed}.  These papers show the convergence of the method by performing numerical simulations on a sequence of uniformly refined meshes. The theoretical convergence of numerical methods for three-phase flows remains an open problem and this paper provides the
	theoretical analysis of DG methods in the case of incompressible three-phase flows under certain conditions on the data.
	While the numerical analysis of three-phase flow is sparse, we note that the case of immiscible two-phase flows in porous media has been investigated
	in several papers.  For instance for incompressible flows, finite difference  methods have been analyzed in \cite{douglas1983finite}, finite volume methods in \cite{ohlberger1997convergence,Eymard2003,Michel03}, DG methods in \cite{EpshteynRiviere2009}, and finite element methods \cite{chen2001error,GiraultRiviereCappanera1,GiraultRiviereCappanera2}.
	
	The paper is organized as follows. In \cref{sec:problem_description}, we present the problem considered and its mathematical formulation. Sections \ref{sec:time_discretization}-\ref{sec:spatial_discretization} describe the time and spatial discretization of our algorithm. Classical projection estimates and the hypothesis used for the numerical analysis of our method are detailed in \cref{sec:preliminaries}. Then we show that the discrete problem is well-posed in \cref{sec:existence_uniqueness} and we establish a priori error estimates in \cref{sec:error_analysis}. Eventually, we perform numerical investigations in \cref{sec:numerical_results} that recover the theoretical rate of convergence for various setups.
	
	%=============================================
	\section{Problem description}
	%=============================================
	\label{sec:problem_description}
	
	Let $p_j$, $s_j$ denote the pressure and the saturation, respectively, of the phase $j$, where $j = \ell, v, a$ (liquid, vapor and aqueous). The saturation for phase $j$ at a point $\bfx$ in the domain $\Omega \subset \mathbb{R}^d$, with $d = 2,3$, is defined as the ratio of the volume of phase $j$ to the total pore volume in a representative elementary volume centered around the point $\bfx$.  Thus, the saturations satisfy
	%
	%\begin{equation}
	%s_\ell  = \frac{\text{volume liquid phase}}{\text{volume pore}}, \quad s_v = \frac{\text{volume vapor phase}}{\text{volume pore}}, \quad s_a = \frac{\text{volume aqueous phase}}{\text{volume pore}},
	%\end{equation}
	%
	%and satisfy the following relation
	%
	\begin{equation}
		s_\ell + s_v + s_a = 1.
	\end{equation}
	Assuming that the phase densities and the porosity are constant, the mass conservation equation of each component is expressed as
	\begin{equation}\label{eq:mass_conserv}
		\phi \partial_t s_j - \nabla \cdot \del[1]{ \kappa\lambda_j \del[0]{\nabla p_j - \rho_j \boldsymbol{g}}} = q_j,\quad j=\ell, a, v,
	\end{equation}
	%		\label{eq:mass_conserv_sl}
	%
	%	\begin{subequations} \label{eq:mass_conserv}
		%		\begin{align}
			%			\partial_t\del[1]{\phi \rho_\ell s_\ell} - \nabla \cdot \del[1]{\rho_\ell \kappa\lambda_\ell \del[0]{\nabla p_\ell - \rho_\ell \boldsymbol{g}}} &= \rho_\ell q_\ell,
			%			\label{eq:mass_conserv_sl}
			%			\\
			%			\partial_t\del[1]{\phi \rho_v s_v} - \nabla \cdot \del[1]{\rho_v \kappa\lambda_v \del[0]{\nabla p_v - \rho_v \boldsymbol{g}}} &= \rho_v q_v,
			%			\label{eq:mass_conserv_sv}
			%			\\
			%			\partial_t\del[1]{\phi \rho_a s_a} - \nabla \cdot \del[1]{\rho_a \kappa\lambda_a \del[0]{\nabla p_a - \rho_a \boldsymbol{g}}} &= \rho_a q_a,
			%			\label{eq:mass_conserv_sa}
			%		\end{align}
		%	\end{subequations}
	%
	where $\kappa$ is the absolute permeability, $\rho_j$ denotes the density of the phase $j$, $\lambda_j$ denotes the mobility of the phase $j$ and $\phi$ is the porosity of the medium. The mobility $\lambda_j$ is defined as $\lambda_j =  k_{rj}/\mu_j$, where $k_{rj}$ and $\mu_j$ represent the relative permeability and viscosity of the phase $j$, respectively. The gravity is denoted by $\boldsymbol{g}$ and $q_\ell$, $q_v$ and $q_a$ are source/sink terms. 
	The differences between phase pressures are capillary pressures  $p_{c,v}$ and $p_{c,a}$ defined as follows:
	\begin{equation}
		p_{c,v} = p_v - p_\ell, \quad p_{c,a} = p_\ell - p_a.
		\label{eq:cap_pressures}
	\end{equation}
	From the set of unknowns (saturations and pressures), we choose for primary unknowns the
	liquid pressure $p_\ell$, the aqueous saturation $s_a$ and the vapor saturation $s_v$. 
	For clarity we explicitly write the dependence of the different quantities with respect to the primary unknowns:
	\begin{align}
		p_{c,v}(s_v), \quad p_{c,a}(s_a), \quad \lambda_\ell(s_v, s_a), \quad \lambda_v(s_v, s_a), \quad \lambda_a(s_v, s_a),\label{eq:depvar1}\\
		\mu_\ell(p_\ell), \quad \mu_v(s_v,s_a), \quad \mu_a(s_v, s_a). \label{eq:depvar2}
	\end{align}
	Moreover, the capillary pressures  are assumed to be differentiable, $\partial_{s_a} p_{c,a}$ is a negative function, and $\partial_{s_v} p_{c,v}$ is a positive function.

	%
	
	%The phase densities $\rho_j$ and phase viscosities $\mu_j$ depend of the phase pressure $p_j$, for $j=v,\ell,a$. The phase relative permeabilities  $k_{rj}$ depend
	%on the phase saturation $s_j$, for $j=v,\ell,a$. The porosity $\phi$ depends on the liquid pressure $p_\ell$.
	
	%---------------------------------------------
	\subsection{Rewritten equations}
	%---------------------------------------------
	
	Summing the three mass conservation equations \eqref{eq:mass_conserv} and using the definition of the capillary pressure \cref{eq:cap_pressures} yields the liquid pressure equation
	%
	%\begin{multline}
	%	\partial_t\del[1]{\phi\del[0]{\rho_\ell s_\ell + \rho_v s_v + \rho_a s_a}} - \nabla\cdot \del[1]{\rho_\ell \lambda_\ell \kappa\del[0]{\nabla p_\ell - \rho_\ell \boldsymbol{g}}}  - \nabla\cdot \del[1]{\rho_v \lambda_v \kappa\del[0]{\nabla p_v - \rho_v \boldsymbol{g}}}  
	%	\\
	%	- \nabla\cdot \del[1]{\rho_a \lambda_a \kappa\del[0]{\nabla p_a - \rho_a \boldsymbol{g}}} = \rho_\ell q_\ell + \rho_v q_v + \rho_a q_a.
	%\end{multline}
	%
	%Using the definition of the capillary pressures \cref{eq:cap_pressures} and rearranging terms, we obtain
	%
	\begin{equation}
		- \nabla\cdot \del[1]{\lambda_t \kappa \nabla p_\ell}  - \nabla\cdot \del[1]{\lambda_v \kappa \nabla p_{c,v}} + \nabla\cdot \del[1]{\lambda_a \kappa\nabla p_{c,a}}
		= q_t - \nabla\cdot \del[1]{\kappa\del[0]{\rho \lambda}_t\boldsymbol{g}},
		\label{eq:liquid_pressure_eq}
	\end{equation}
	where
	\begin{equation}
		\del[0]{\rho\lambda}_t = \rho_\ell\lambda_\ell + \rho_v\lambda_v + \rho_a\lambda_a, \quad \lambda_t = \lambda_\ell + \lambda_v + \lambda_a, \quad q_t = q_\ell + q_v + q_a.
	\end{equation}
	%
	%Next we expand the time derivative term. Note that, using the chain rule, we have
	%	\begin{equation}
		%		\psi_\ell = \partial_{p_\ell} \del[0]{\phi\rho_t} = \rho_t \partial_{p_\ell}\phi + \phi \del{s_\ell \partial_{p_\ell}\rho_\ell  + s_a \partial_{p_\ell}\rho_a + s_v \partial_{p_\ell}\rho_v},
		%	\end{equation}
	%	%
	%	and
	%	%
	%	\begin{equation}
		%		\nu_{\ell,j} = -\partial_{s_j}(\phi\rho_t) = \phi\del{\rho_\ell - \rho_j - s_j \partial_{s_j}\rho_j}, \quad j = a,v.
		%	\end{equation}
	%
	Using the capillary pressure $p_{c,a}$, the mass conservation \eqref{eq:mass_conserv} satisfied by the aqueous saturation can be rewritten
	%
	%\begin{equation}
	%\nabla p_{c,a} = \partial_{s_v} p_{c,a}\nabla s_v + \partial_{s_a} p_{c,a}\nabla s_a,
	%\end{equation}
	%%
	%the equation \eqref{eq:mass_conserv_sa} can be rewritten
	%
	\begin{equation}
		\phi \partial_t s_a + \nabla\cdot \del{\kappa\lambda_a \partial_{s_a} p_{c,a}\nabla s_a} - \nabla\cdot\del[1]{\kappa\lambda_a \nabla p_\ell} = q_a
		- \nabla\cdot\del[1]{\rho_a\kappa\lambda_a \boldsymbol{g}}.
		\label{eq:aqueous_saturation_eq}
	\end{equation}
	Similarly, the vapor saturation $s_v$ satisfies the following equation, derived from \eqref{eq:mass_conserv} with $j=v$. 
	%Using the relations $p_v = p_{c,v} + p_\ell$, and
	%
	%\begin{equation}
	%\nabla p_{c,v} = \partial_{s_v} p_{c,v}\nabla s_v,
	%\end{equation}
	%the equation \eqref{eq:mass_conserv_sv} can be rewritten as follows
	%
	\begin{equation}
		\phi \partial_t s_v - \nabla\cdot \del{\kappa\lambda_v \partial_{s_v} p_{c,v}\nabla s_v}  
		- \nabla\cdot\del[1]{\kappa\lambda_v \nabla p_\ell} = q_v
		- \nabla\cdot\del[1]{\rho_v\kappa\lambda_v \boldsymbol{g}}.
		\label{eq:vapor_saturation_eq}
	\end{equation}

	These equations are complemented with Dirichlet and Neumann boundary conditions. The boundary of the computational domain $\Omega$ is decomposed as
	\begin{equation}
		\partial\Omega = \Gamma_{\text{D}}^{p_\ell} \cup \Gamma_{\text{N}}^{p_\ell} = \Gamma_{\text{D}}^{s_a} \cup \Gamma_{\text{N}}^{s_a} = \Gamma_{\text{D}}^{s_v} \cup \Gamma_{\text{N}}^{s_v},
	\end{equation}
	with $\vert \Gamma_{\text{D}}^{p_\ell}\vert >0, \vert \Gamma_{\text{D}}^{s_a}\vert > 0, \vert \Gamma_{\text{D}}^{s_v} \vert > 0$. 
	The Dirichlet boundary conditions imposed on $\Gamma_{\text{D}}^{p_\ell}$, $\Gamma_{\text{D}}^{s_a}$ and $\Gamma_{\text{D}}^{s_v}$ are denoted by $p_\ell^{\text{bdy}}$, $s_a^{\text{bdy}}$, $s_v^{\text{bdy}}$. The Neumann boundary conditions imposed on $\Gamma_{\text{N}}^{p_\ell}$, $\Gamma_{\text{N}}^{s_a}$ and $\Gamma_{\text{N}}^{s_v}$ are given by
	\begin{subequations}
		\begin{equation}
			\del{\lambda_t \kappa \nabla p_\ell + 	\lambda_v\kappa \nabla p_{c,v} - \lambda_a\kappa \nabla p_{c,a} - \kappa (\rho\lambda)_t \boldsymbol{g}}\cdot \boldsymbol{n} = j_p^{\text{N}},
			\label{eq:Neumann_BC_pressure}
		\end{equation}
		\begin{equation}
			\del{- \kappa \lambda_a \partial_{s_a} p_{c,a}\nabla s_a + \kappa\lambda_a\nabla p_\ell - \rho_a \kappa \lambda_a \boldsymbol{g}}\cdot \boldsymbol{n} = j_{s_a}^{\text{N}},
			\label{eq:Neumann_BC_aqueous_saturation}
		\end{equation}
		\begin{equation}
			\del{\kappa \lambda_v \partial_{s_v} p_{c,v} \nabla s_v + \kappa\lambda_v\nabla p_\ell - \rho_v \kappa \lambda_v \boldsymbol{g}}\cdot \boldsymbol{n} = j_{s_v}^{\text{N}},
			\label{eq:Neumann_BC_vapor_saturation}
		\end{equation}
	\end{subequations}
	where $\boldsymbol{n}$ represents the outward unit normal vector to the boundary $\partial\Omega$. 
	
	%=============================================
	\section{Time discretization}
	%=============================================
	\label{sec:time_discretization}
	
	For the time discretization, we use a backward Euler method and partition the time interval $[0, T]$ using a time step $\tau > 0$ such that $N\tau = T$.  In the rest of the paper we define $t_n = n\tau$ for any integer $0 \leq n \leq N$, and for any time dependent function $f$, we define $f^n = f|_{t = t_n}$.
	
	%---------------------------------------------
	\subsection{Liquid pressure}
	%---------------------------------------------
	
	The time discretization of the liquid pressure \cref{eq:liquid_pressure_eq} reads
	\begin{equation}
		- \nabla\cdot \del[1]{\lambda_t^n \kappa\nabla p_\ell^{n+1}}
		= q_t^{n+1} 
		- \nabla\cdot \del[1]{\kappa\del[0]{\rho\lambda}_t^n\boldsymbol{g}} 
		+ \nabla\cdot \del[1]{\lambda_v^n \kappa\nabla p_{c,v}^n} - \nabla\cdot \del[1]{\lambda_a^n \kappa\nabla p_{c,a}^n}.
	\end{equation}
	%
	
	%---------------------------------------------
	\subsection{Aqueous saturation}
	%---------------------------------------------
	
	The time discretization of the aqueous saturation equation \cref{eq:aqueous_saturation_eq} is
	\begin{equation}
		\phi \frac{s_a^{n+1} - s_a^n}{\tau} + \nabla\cdot \del{\kappa\lambda_a^n \del[1]{\partial_{s_a}p_{c,a}}^n \nabla s_a^{n+1}} = q_a^{n+1} + \nabla\cdot\del[1]{\kappa\lambda_a^n \del[0]{\nabla p_\ell^{n+1} - \rho_a\boldsymbol{g}}}.
		\label{eq:satu_aux}
	\end{equation}
	Note that $\partial_{s_a} p_{c,a}$ is negative. Therefore, with $\del[1]{\partial_{s_a}p_{c,a}}^{n,+} = -\del[1]{\partial_{s_a}p_{c,a}}^{n}$, we may write \cref{eq:satu_aux} as
	\begin{equation}
		\phi \frac{s_a^{n+1} - s_a^n}{\tau} - \nabla\cdot \del{\kappa\lambda_a^n \del[1]{\partial_{s_a}p_{c,a}}^{n,+} \nabla s_a^{n+1}} = q_a^{n+1} + \nabla\cdot\del[1]{\kappa\lambda_a^n \del[0]{\nabla p_\ell^{n+1} - \rho_a\boldsymbol{g}}}.
		\label{eq:satu_aux2}
	\end{equation}
	%
	%where $\del[1]{\partial_{s_a}p_{c,a}}^{n,+} = -\del[1]{\partial_{s_a}p_{c,a}}^{n}$.

	%--------------------------------------
	\subsection{Vapor saturation}
	%--------------------------------------
	
	The time discretization of the vapor saturation equation \cref{eq:vapor_saturation_eq} reads
	\begin{equation}
		\phi \frac{s_v^{n+1} - s_v^n}{\tau} - \nabla\cdot \del{\kappa\lambda_v^n \del[1]{\partial_{s_v}p_{c,v}}^n \nabla s_v^{n+1}} = q_v^{n+1}
		+ \nabla\cdot\del[1]{\kappa\lambda_v^n \del[0]{\nabla p_\ell^{n+1} - \rho_v\boldsymbol{g}}}.
		\label{eq:satu_aux_vap}
	\end{equation}
	%
	%After obtaining $s_a^{n+1}, s_v^{n+1}$,  we can update the coefficients that depend on the unknowns.
	%the other phase variables $\lambda_v^{n+1}$, $s_\ell^{n+1}$ and $\lambda_\ell^{n+1}$ can be computed.
	
	%=============================================
	\section{Spatial discretization}
	%=============================================
	\label{sec:spatial_discretization}
	
	For the spatial discretization, we use an interior penalty discontinuous Galerkin method. The domain $\Omega$ is discretized with a conforming, shape-regular mesh $\mathcal{E}_h$ consisting of simplices or
	quadrilateral and hexaedral elements. We denote by $h_e$ and $h_K$ the size of an edge (or face for $d=3$)  $e$ and an element $K$, respectively. Moreover, we define the mesh size $h = \max_{K \in \mathcal{E}_h}h_K$. For any quadrilateral element $K$, we define the two-dimensional local polynomial space $\mathbb{P}_{k_1,k_2}(K)$ as
	\begin{equation}
		\mathbb{P}_{k_1,k_2}(K) = \cbr{p(x,y) \; | \; p(x,y) = \sum_{i \leq k_1}\sum_{j \leq k_2}a_{ij}x^i y^j}.
	\end{equation}
	The three-dimensional local polynomial space $\mathbb{P}_{k_1,k_2,k_3}(K)$ is defined similarly. Finally, we define $\mathbb{Q}_k(K) = \mathbb{P}_{k,k}(K)$ for $d = 2$, and $\mathbb{Q}_k(K) = \mathbb{P}_{k,k,k}(K)$ for $d = 3$.
	The space of discontinuous piecewise linear polynomials is denoted by $X_h$. If $\mathcal{E}_h$ consists of quadrilateral or hexaedral elements, the space $X_h$ is defined by:
	\begin{equation}
		X_h = \cbr[1]{v \in L^2(\Omega) : v|_K \in \mathbb{Q}_{1}(K), \forall K \in \mathcal{E}_h}.
	\end{equation}

	The discrete liquid pressure, aqueous saturation and vapor saturation at time $t_n$ are denoted by
	$P_h^n, S_{a_h}^n$ and $S_{v_h}^n$ respectively; they belong to the finite-dimensional spaces $X_h$.
	The Dirichlet boundary conditions are imposed strongly; thus we assume that the data $p_\ell^{\mathrm{bdy}}, s_a^{\mathrm{bdy}}, s_v^{\mathrm{bdy}}$ are traces of functions in $X_h$.  This assumption is in agreement with realistic simulations where the Dirichlet data are
	simply constants on the Dirichlet boundaries.  We will make use of the following finite-dimensional spaces for the test functions:
	\begin{equation}
		X_{h,\Gamma_\D^{p_\ell}} = X_h \cap \{ v=0 \mbox{ on } \Gamma_\D^{p_\ell}\}, \quad
		X_{h,\Gamma_\D^{s_a}} = X_h \cap \{ v=0 \mbox{ on } \Gamma_\D^{s_a}\}, \quad
		X_{h,\Gamma_\D^{s_v}} = X_h \cap \{v=0 \mbox{ on } \Gamma_\D^{s_v}\}.
	\end{equation}
	We also define the Raviart--Thomas space $\mathbb{RT}_0$:
	\begin{equation}
		\mathbb{RT}_0 = \cbr[1]{\boldsymbol{u} \in H(\text{div},\Omega) : \boldsymbol{u}|_K \in \mathbb{RT}_{0}(K), \forall K \in \mathcal{E}_h},
	\end{equation}
	where
	\begin{equation}
		\mathbb{RT}_0(K) = \begin{cases}
			\mathbb{P}_{1,0}(K) \times \mathbb{P}_{0,1}(K) & d = 2,
			\\
			\mathbb{P}_{1,0,0}(K) \times \mathbb{P}_{0,1,0}(K) \times \mathbb{P}_{0,0,1}(K) & d = 3.
		\end{cases}
	\end{equation}
	We note that the above spaces can be defined similarly if one uses simplices elements. 
	The set of interior faces is denoted by $\Gamma_h$.  For any interior face $e$, we fix a unit normal vector $\boldsymbol{n}_e$, and we denote by $K_1$ and $K_2$ the elements that share
	the face $e$ such that $\boldsymbol{n}_e$ points from $K_1$ into $K_2$. 
	For any function  $f\in X_h$,  we define the jump operator $\sbr[0]{\cdot}$ on interior faces as $\sbr[1]{f} = f_1 - f_2$, where $f_i = f|_{K_i}$. Moreover, we define the weighted average operator $\cbr[0]{\cdot}$ on interior faces as $\cbr[0]{A \nabla f \cdot \boldsymbol{n}_e} = \omega_1 A_1 \nabla f_1 \cdot \boldsymbol{n}_e + \omega_2 A_2 \nabla f_2 \cdot \boldsymbol{n}_e$, where $\omega_1 = A_2\del[0]{A_1 + A_2}^{-1}$ and $\omega_2 = A_1\del[0]{A_1 + A_2}^{-1}$. Note that the standard average operator with weights $\omega_1 = \omega_2 = 1/2$
	is denoted by $\{\cdot\}_\frac12$. On boundary faces, the jump and weighted average operators are defined as follows $\sbr[0]{f}=\{f\}=f  $. In the following, the $L^2$ inner-product over $\Omega$ is denoted by $(\cdot,\cdot)$. The parameters $\theta_{p_\ell}, \, \theta_{s_a},\, \theta_{s_v}$ take values $-1,\,0,\,1$ which respectively correspond to symmetric, incomplete and nonsymmetric interior penalty discontinuous Galerkin.

	%We define penalty parameters $\gamma_{p_\ell,e}$, $\gamma_{s_a,e}$ and $\gamma_{s_v,e}$ as
	%%
	%\begin{equation}
	%	\gamma_{p_\ell,e} = \eta_{p_\ell}\gamma_{A,e}, \quad \gamma_{s_a,e} = \eta_{s_a}\gamma_{A,e}, \quad \gamma_{s_v,e} = \eta_{s_v}\gamma_{A,e},
	%\end{equation}
	%%
	%where $\eta_{p_\ell}$, $\eta_{s_a}$ and $\eta_{s_v}$ are positive constants, and 
	%%
	%\begin{equation}
	%	\gamma_{A,e} = \frac{2 A_1 A_2}{A_1 + A_2}.
	%\end{equation}
	%
	
	%---------------------------------------------
	\subsection{Liquid pressure}
	%---------------------------------------------
	
	The discrete problem for the liquid pressure reads: find $P_h^{n+1} \in X_h$ such that $P_h^{n+1} = p_\ell^\mathrm{bdy}$  on $\Gamma_\D^{p_\ell}$ and the following relation is satisfied for all $w_h \in X_{h,\Gamma_\D^{p_\ell}}$:
	\begin{equation}
		b_{p}^{n}(P_{h}^{n+1},w_h) = f_{p}^{n}(w_h),
		\label{eq:disc_pl}
	\end{equation}
	where, $b_{p}^{n}(P_{h}^{n+1},w_h) = b_{p}(P_{h}^{n+1},w_h; P_{h}^{n}, S_{a_h}^n, S_{v_h}^n)$, $f_{p}^{n}(w_h) = 
	f_{p}(w_h; P_{h}^{n}, S_{a_h}^n, S_{v_h}^n)$, with $b_{p}$ and $f_{p}$ defined as
	\begin{equation}
		\begin{split}
			b_{p}(v_h,w_h; P_{h}^n, S_{a_h}^n, S_{v_h}^n) &= \sum_{K \in \mathcal{E}_h}\int_K \lambda_t^n\kappa \nabla v_{h}\cdot \nabla w_h
			+ \sum_{e \in \Gamma_h} \alpha_{p_\ell,e} h_e^{-1}\int_e \eta_{p_\ell,e}^n \sbr[1]{v_{h}}\sbr[1]{w_h}
			\\
			&-\sum_{e \in \Gamma_h }\int_e \cbr[1]{\lambda_t^n \kappa\nabla v_h\cdot \boldsymbol{n}_e}\sbr[1]{w_h} 
			+ \theta_{p_\ell} \sum_{e \in \Gamma_h }\int_e \cbr[1]{\lambda_t^n \kappa\nabla w_h \cdot\boldsymbol{n}_e}\sbr[1]{v_{h}},
		\end{split}
		\label{eq:ah_pressure}
	\end{equation}
	and 
	\begin{equation}
		\begin{split}
			f_{p}(w_h;&P_{h}^n, S_{a_h}^n, S_{v_h}^n) = ( q_t^{n+1}, w_h)
			+\sum_{e \in \Gamma_{\text{N}}^{p_\ell}}\int_e j_p^{\text{N}} w_h 
			\\
			&- \sum_{K \in \mathcal{E}_h}\int_K \del[1]{ \lambda_v^n \kappa\nabla p_{c,v}^n - \lambda_a^n \kappa\nabla p_{c,a}^n - \kappa\del[0]{\rho \lambda}_t^n\boldsymbol{g} }\cdot \nabla w_h
			\\
			&+\sum_{e \in \Gamma_h }\int_e \cbr[1]{ \lambda_v^n \kappa\nabla p_{c,v}^n\cdot \boldsymbol{n}_e}\sbr[1]{w_h} 
			- \sum_{e \in \Gamma_h}\int_e \cbr[1]{\lambda_a^n \kappa\nabla p_{c,a}^n \cdot \boldsymbol{n}_e}\sbr[1]{w_h}
			\\
			&- \sum_{e \in \Gamma_h}\int_e \cbr[1]{\kappa\del[0]{\rho\lambda}_t^n\boldsymbol{g} \cdot \boldsymbol{n}_e}\sbr[1]{w_h}.
		\end{split}
		\label{eq:fh_pressure}
	\end{equation}
	%
	%In the above, we have used the short-hand notation
	%\[
	%\nabla p_{c,v}^n = \partial_{s_v}p_{c,v}^n \nabla S_{v_h}^n, \quad
	%\nabla p_{c,a}^n = \partial_{s_v}p_{c,a}^n \nabla S_{v_h}^n.
	%\]
	%\gsj{Check if we want to get rid off the equation above.}
	We recall that $\lambda_t^n, \,(\rho \lambda)_t^n, \, \lambda_i^n$ for $i=v,\ell,a$ are the functions
	$\lambda_t, \, (\rho\lambda)_t, \, \lambda_i$ evaluated at the discrete solutions (discrete pressures and saturations) at time $t_n$.
	The penalty parameter $\alpha_{p_\ell,e}$ is a positive constant such that $0 < \alpha_{p_\ell,*} \leq \alpha_{p_\ell,e} \leq \alpha_{p_\ell}^*$, and the penalty parameter $\eta_{p_\ell,e}^n$ depends on the absolute permeability and mobilities in the following way:
	\begin{equation}
		\eta_{p_\ell,e}^n = \mathcal{H}\del{\del{\kappa \lambda_t^n}|_{K_1}, \del{\kappa \lambda_t^n}|_{K_2}}, \quad
		\forall e=\partial K_1\cap\partial K_2, %\frac{2\kappa_1\del[1]{\rho\lambda}_{t,1}\kappa_2\del[1]{\rho\lambda}_{t,2}}{\kappa_1\del[1]{\rho\lambda}_{t,1} + \kappa_2\del[1]{\rho\lambda}_{t,2}}.
	\end{equation}
	where $\mathcal{H}$ is the harmonic average function:
	\begin{equation}
		\mathcal{H}(x_1,x_2) = \frac{2 x_1 x_2}{x_1+x_2}.
		\label{eq:harmonic}
	\end{equation}
	
	%---------------------------------------------
	\subsection{Aqueous saturation}
	%---------------------------------------------
	
	The discrete problem for the aqueous saturation reads: find $S_{a_h}^{n+1} \in X_h$ such that $S_{a_h}^{n+1} = s_a^\mathrm{bdy}$ on $\Gamma_\D^{s_a}$ and such that the following relation is satisfied for all $w_h \in X_{h,\Gamma_\D^{s_a}}$:
	\begin{equation}
		\frac{1}{\tau}(\phi S_{a_h}^{n+1},w_h) + b_{a}^{n}(S_{a_h}^{n+1},w_h) \\
		= \frac{1}{\tau} (\phi S_{a_h}^{n},w_h) + f_a^{n}(w_h),
		\label{eq:disc_Sa}
	\end{equation}
	where, $b_{a}^{n}(S_{a_h}^{n+1},w_h) = b_{a}(S_{a_h}^{n+1},w_h; P_{h}^{n+1}, S_{a_h}^n, S_{v_h}^n)$, $f_{a}^{n}(w_h) = 
	f_{a}(w_h; P_{h}^{n+1}, S_{a_h}^n, S_{v_h}^n)$, with $b_{a}$ and $f_{a}$ defined as
	\begin{equation}
		\begin{split}
			b_{a}(v_h,w_h;&P_{h}^{n+1}, S_{a_h}^n, S_{v_h}^n) = \sum_{K \in \mathcal{E}_h}\int_K \kappa\lambda_a^n \del[1]{\partial_{s_a}p_{c,a}}^{+,n} \nabla v_h\cdot \nabla w_h 
			\\
			&-\sum_{e \in \Gamma_h }\int_e \cbr{\kappa \lambda_a^n \del[1]{\partial_{s_a}p_{c,a}}^{+,n} \nabla v_h \cdot \boldsymbol{n}_e}\sbr[1]{w_h} 
			+ \sum_{e \in \Gamma_h }\alpha_{s_a,e}h_e^{-1}\int_e \eta_{s_a,e}^n \sbr[1]{v_h} \sbr[1]{w_h} 
			\\
			&+ \theta_{s_a} \sum_{e \in \Gamma_h }\int_e \cbr{\kappa \lambda_a^n \del[1]{\partial_{s_a}p_{c,a}}^{+,n} \nabla w_h \cdot \boldsymbol{n}_e}\sbr[1]{v_h},
		\end{split}
		\label{eq:ah_aqueous_saturation}
	\end{equation}
	and
	\begin{equation}
		\label{eq:fadef}
		\begin{split}
			f_a(w_h;&P_{h}^{n+1}, S_{a_h}^n, S_{v_h}^n) = (q_a^{n+1}, w_h) +
			\sum_{K \in \mathcal{E}_h}\int_K \del{\lambda_a^n \boldsymbol{u}_h^{n+1} + \kappa \rho_a \lambda_a^n \boldsymbol{g}} \cdot \nabla w_h 
			+ \sum_{e \in \Gamma_{\text{N}}^{s_a}}\int_e j_{s_a}^{\text{N}} w_h 
			\\
			&- \sum_{e \in \Gamma_h }\int_e \left( \lambda_a^n\right)^{\uparrow}_{s_a} \boldsymbol{u}_h^{n+1} \cdot \boldsymbol{n}_e  \sbr[1]{w_h} 
			- \sum_{e \in \Gamma_h }\int_e \cbr{ \rho_a\kappa \lambda_a^n\boldsymbol{g}\cdot\boldsymbol{n}_e} \sbr[1]{w_h}.%\\
			%&- \Gr \sum_{e \in \Gamma_h \cup \Gamma_{\text{D}}^{s_a}}\int_e \left(\frac{\rho_a^{\ast,n+1} \lambda_a^n}{\del[0]{\rho\lambda}_t^n}\right)^{\uparrow}_{s_a} \boldsymbol{u}_h\cdot\boldsymbol{n}_e  \sbr[1]{w_h}.\\
		\end{split}
	\end{equation}
	%
	%\[
	%\rho_a^{\ast,n+1} = \rho_a(P_{h}^{n+1}-p_{c,a}^n).
	%\]
	%\Rd (BR: To be precise, $\rho_a^n$ is the function $\rho_a$ evaluated everywhere at $t_n$, whereas $\rho_a^{\ast,n+1}$ is evaluated for some
	%terms at $t_{n+1}$ and for some terms at $t_n$ ) \Bk
	In \eqref{eq:fadef}, the vector  $\boldsymbol{u}_h^{n+1}$ is the projection of the approximation of the
	Darcy velocity onto the Raviart--Thomas space $\mathbb{RT}_0$ (see the exact definition of operator  $\Pi_\mathrm{RT}$ in Section~\ref{sec:RT}):  
	\[
	\boldsymbol{u}_h^{n+1} = \Pi_{\mathrm{RT}}(-\kappa \nabla P_h^{n+1}).
	\]
	The upwind operator $(\cdot)^{\uparrow}_{s_a}$ is defined as follows. For readibility, let $D = \lambda_a^n$ and $D^g = \rho_a\kappa \lambda_a^n$. For an interior edge $e$ shared by two elements $K_1$ and $K_2$, we have %\Rd Q: since $\boldsymbol{u}_h$ is in the Raviart-Thomas space, its average $\{\boldsymbol{u}_h\cdot\bfn_e\}=\boldsymbol{u}_h\cdot\bfn_e$ coincides with itself, so we do not need to use the average notation??\Bk
	\begin{equation}
		(D)^{\uparrow}_{s_a}  = \begin{cases}
			D|_{K_1} & \text{ if } \{ D \boldsymbol{u}_h^{n+1}  + D^g \boldsymbol{g} \}_\frac12 \cdot \bfn_e  \geq 0,
			\\
			D|_{K_2} & \text{ otherwise }.
		\end{cases}
	\end{equation}
	%\begin{equation}
	%\cbr[1]{D \boldsymbol{u}_h\cdot\boldsymbol{n}_e}^{\uparrow}_{s_a} = \begin{cases}
		%\frac{1}{2}D|_{K_1}\del[1]{\boldsymbol{u}|_{K_1} + \boldsymbol{u}|_{K_2}} \cdot \boldsymbol{n}_e & \text{ if } \xi \geq 0
		%\\
		%\frac{1}{2}D|_{K_2}\del[1]{\boldsymbol{u}|_{K_1} + \boldsymbol{u}|_{K_2}} \cdot \boldsymbol{n}_e & \text{ if } \xi < 0,
		%\end{cases}
		%\end{equation}
		%
		%where $\xi = \del[1]{\del[1]{D\boldsymbol{u}_h^n + D^g\boldsymbol{g}}|_{K_1} + \del[1]{D\boldsymbol{u}_h^n + D^g\boldsymbol{g}}|_{K_2}}\cdot\boldsymbol{n}_e$. For a boundary edge $e\subset\partial K$, $(D)^{\uparrow}_{s_a} = D|_{K}$.
		
		The penalty parameter $\alpha_{s_a,e}$ is a positive constant such that $0 < \alpha_{s_a,*} \leq \alpha_{s_a,e} \leq \alpha_{s_a}^*$, and the parameter $\eta_{s_a,e}^n$ is defined on the interior faces by
		\begin{equation}
			\eta_{s_a,e}^n = \mathcal{H}\del{\del{\kappa(\partial_{s_a} p_{c,a})^{+,n} \lambda_a^n}|_{K_1}, \del{\kappa(\partial_{s_a} p_{c,a})^{+,n}\lambda_a^n}|_{K_2}}, \quad \forall e = \partial K_1\cap\partial K_2.
		\end{equation}
		%
		
		%---------------------------------------------
		\subsection{Vapor saturation}
		%---------------------------------------------
		
		The discrete problem for the vapor saturation reads: find $S_{v_h}^{n+1} \in X_h$ such that $S_{v_h}^{n+1} = s_v^\mathrm{bdy}$ on $\Gamma_\D^{s_v}$ and such that the following relation is satisfied for all $w_h \in X_{h,\Gamma_\D^{s_v}}$:
		\begin{equation}
			\frac{1}{\tau}(\phi S_{v_h}^{n+1},w_h) + b_{v}^{n}(S_{v_h}^{n+1},w_h) \\
			= \frac{1}{\tau}(\phi S_{v_h}^{n},w_h) + f_{v}^{n}(w_h),
			\label{eq:disc_Sv}
		\end{equation}
		where $b_{v}^{n}(S_{v_h}^{n+1},w_h) = b_{v}(S_{v_h}^{n+1},w_h;P_{h}^{n+1}, S_{a_h}^{n+1}, S_{v_h}^n)$, $f_{v}^{n}(w_h) = f_{v}(w_h;P_{h}^{n+1}, S_{a_h}^{n+1}, S_{v_h}^n)$, with $b_{v}$ and $f_{v}$ defined as
		\begin{equation}
			\begin{split}
				b_{v}(v_h,w_h;&P_{h}^{n+1}, S_{a_h}^{n+1}, S_{v_h}^n) = \sum_{K \in \mathcal{E}_h}\int_K \kappa\lambda_v^n \partial_{s_v}p_{c,v}^n \nabla v_h\cdot \nabla w_h 
				\\
				&-\sum_{e \in \Gamma_h }\int_e \cbr{\kappa \lambda_v^n \partial_{s_v}p_{c,v}^n \nabla v_h \cdot \boldsymbol{n}_e}\sbr[1]{w_h} 
				+ \sum_{e \in \Gamma_h }\alpha_{s_v,e}h_e^{-1}\int_e \eta_{s_v,e}^n\sbr[1]{v_h} \sbr[1]{w_h} 
				\\
				&+ \theta_{s_v} \sum_{e \in \Gamma_h }\int_e \cbr{\kappa \lambda_v^n \partial_{s_v}p_{c,v}^n \nabla w_h \cdot \boldsymbol{n}_e}\sbr[1]{v_h},
			\end{split}
			\label{eq:ah_vapor_saturation}
		\end{equation}
		and
		\begin{equation}
			\begin{split}
				f_{v}(w_h;&P_{h}^{n+1}, S_{a_h}^{n+1}, 	S_{v_h}^n) = (q_v^{n+1}, w_h) + \sum_{K \in \mathcal{E}_h}\int_K \del{\lambda_v^n \boldsymbol{u}_h^{n+1} + \kappa \rho_v \lambda_v^n \boldsymbol{g}}\cdot \nabla w_h 
				\\
				&+ \sum_{e \in \Gamma_{\text{N}}^{s_v}}\int_e j_{s_v}^{\text{N}} w_h 
				- \sum_{e \in \Gamma_h }\int_e 	\left( \lambda_v^n\right)^\uparrow_{s_v}  \boldsymbol{u}_h^{n+1}\cdot\boldsymbol{n}_e \sbr[1]{w_h} 
				- \sum_{e \in \Gamma_h }\int_e \cbr{\rho_v\kappa \lambda_v^n\boldsymbol{g}\cdot\boldsymbol{n}_e} \sbr[1]{w_h},
			\end{split}
		\end{equation}
		where $(\cdot)^{\uparrow}_{s_v}$ denotes the upwind average operator which is defined similarly as $(\cdot)^{\uparrow}_{s_a}$, but with $D = \lambda_v^n$ and $D^g = \rho_v \kappa \lambda_v^n$. The penalty parameter $\alpha_{s_v,e}$ is a positive constant such that $0 < \alpha_{s_v,*} \leq \alpha_{s_v,e} \leq \alpha_{s_v}^*$, and $\eta_{s_v,e}^n$ is defined by
		\begin{equation}
			\eta_{s_v,e}^n  = \mathcal{H}\del{\del{\kappa (\partial_{s_v}p_{c,v})^n \lambda_v^n}|_{K_1}, \del{\kappa (\partial_{s_v}p_{c,v})^n \lambda_v^n}|_{K_1}}.
		\end{equation}
		%	
		
		%--------------------------------------
		\subsection{Starting the algorithm}
		%--------------------------------------
		To start the algorithms, we choose the $L^2$ projections of the unknowns at time $t_0$. Let $\Pi_h$ be the
		$L^2$ projection onto $X_h$. 
		\begin{equation}\label{eq:L2projinit}
			P_h^0 = \Pi_{h} p_\ell^0, \quad
			S_{a_h}^0 = \Pi_{h} s_a^0, \quad
			S_{v_h}^0 = \Pi_{h} s_v^0,
		\end{equation}
		where $p_\ell^0, s_a^0, s_v^0$ are the exact solutions at time $t_0$.
		
		%---------------------------------------------
		\subsection{Raviart--Thomas projection}
		\label{sec:RT}
		%---------------------------------------------
		
		The Raviart-Thomas projection, $\bfu_h^{n+1} = \Pi_\mathrm{RT}(- \kappa  \nabla P_h^{n+1})$ is defined by the following equations:
		%Once $P_h^{n+1}$ has been computed, we obtain the total discrete velocity $\boldsymbol{u}_h^n$ by projecting  $- \kappa \overline{(\rho\lambda)}_t^n \nabla P_h^{n+1}$ onto the Raviart--Thomas space $\mathbb{RT}_0$. This projection was introduced for elliptic PDEs in~\cite{Ern:2007}. 
		%The projection $\boldsymbol{u}_h$ is defined locally by
		%
		\begin{subequations}
			\begin{align}
				\int_e \boldsymbol{u}_h^{n+1} \cdot \boldsymbol{n}_e q_h &= -\int_e \cbr[0]{\kappa \nabla P_h^{n+1}\cdot \boldsymbol{n}_e} q_h + \alpha_{p_\ell,e} h_e^{-1} \int_e \eta_{p_\ell,e}^n \sbr[0]{P_h^{n+1}} q_h, & \forall q_h \in \mathbb{Q}_0(e), \forall e \in \Gamma_h, \label{eq:RTproj1}
				\\
				%\int_e \boldsymbol{u} \cdot \boldsymbol{n}_e q_h &= -\int_e c_h \nabla P_h\cdot \boldsymbol{n}_e q_h + \int_e \eta \gamma_{c_h,e} h_e^{-1} \del[0]{P_h - p_\ell^{\text{bdy}}} q_h, & \forall q_h \in \mathbb{Q}_0(e), \forall e \in \Gamma_{\text{D}}^{p_\ell},
				%\\
				\int_e \boldsymbol{u}_h^{n+1} \cdot \boldsymbol{n}_e q_h &= -\int_e \kappa \nabla P_h^{n+1}\cdot \boldsymbol{n}_e q_h, &\forall q_h \in \mathbb{Q}_0(e), \forall e \in  \partial\Omega. 
				\label{eq:RTproj2}
			\end{align}
		\end{subequations}
		This projection was introduced for elliptic PDEs in~\cite{Ern:2007} for spaces of the same order; we apply here to Raviart-Thomas spaces with a degree less than the DG spaces.
		
		%
		%=============================================
		\section{Preliminaries}
		%=============================================
		\label{sec:preliminaries}
		
		In this section, we establish some notation and recall some well-known results from finite element analysis that will be used in the rest of the paper. Finally, we list the hypotheses assumed in this work.
		
		%--------------------------------------
		\subsection{Notation and useful results}
		%--------------------------------------
		
		The L$^2$ norm over a set $D$ is denoted by $\Vert \cdot\Vert_{L^2(D)}$. When $D = \Omega$, the subscript will be omitted. Let us define the space $X(h) = X_h + H^2(\Omega)$. For functions $w \in X(h)$, we define the broken gradient $\nabla_h w$ by $(\nabla_h w)|_K = \nabla(w|_K)$. The space $X(h)$ is endowed with the \textit{coercivity} norm for all $w \in X(h)$ 
		\begin{equation}
			\vertiii{w} := \del{\norm[1]{\nabla_h w}^2 + \abs[1]{w}_{\text{J}}^2}^{1/2}, \quad
			\abs[1]{w}_{\text{J}} = \del{\sum_{e \in \Gamma_h }h_e^{-1}\norm[1]{\sbr[0]{w}}_{L^2(e)}^2}^{1/2}.
			\label{eq:dif_norm}
		\end{equation}
		Additionally, we introduce the following norm on $X(h)$:
		\begin{equation}
			\vertiii{w}_{*} := \del{\vertiii{w}^2 + \sum_{K \in \mathcal{E}_h}h_K \norm[1]{\nabla_h w|_K\cdot \boldsymbol{n}_K}_{L^2(\partial K)}^2 }^{1/2}.
			\label{eq:stab_norm}
		\end{equation}

		The following classical finite element results will be used in the analysis carried out in \cref{sec:existence_uniqueness} and \cref{sec:error_analysis}.
		\begin{lemma}[Trace inequality]
			Let $\mathcal{E}_h$ be a shape-regular mesh with parameter $C_{\text{shape}}$. Then, for all $w_h \in X_h$, all $K \in \mathcal{E}_h$ and all $e \in \partial K$, we have
			\begin{equation}
				\norm[1]{w_h}_{L^2(e)} \leq C_{\text{tr}}h_K^{-1/2} \norm[1]{w_h}_{L^2(K)},
				\label{eq:trace_ineq}
			\end{equation}
			where $C_{\text{tr}} > 0$ depends only on $C_{\text{shape}}$.
		\end{lemma}
		\begin{lemma}[Discrete Poincar\'e inequality \cite{Brenner:2003}]
			For all $w$ in the broken Sobolev space $H^1(\mathcal{E}_h)$, there exists a constant $C_P > 0$ independent of $h$ such that
			\begin{equation}
				\norm[1]{w} \leq C_P \vertiii{w}.
				\label{eq:poincare}
			\end{equation}
		\end{lemma}

		We denote by $\pi_{h,\Gamma}$ the $L^2$-orthogonal projection onto $X_{h,\Gamma}$ for $\Gamma \in \{\Gamma_\D^{p_\ell}, \Gamma_\D^{s_a}, \Gamma_\D^{s_v}\}$. The following lemma recalls approximation estimates that are later used in the analysis of the numerical scheme introduced in \cref{sec:spatial_discretization}.
		\begin{lemma}[$L^2$-orthogonal projection approximation bounds]
			For any element $K \in \mathcal{E}_h$, for all $s \in \cbr[1]{0,1,2}$ and all $w \in H^s(K)$, there holds
			\begin{equation}
				\abs[1]{w - \pi_{h,\Gamma}w}_{H^m(K)} \leq C h_K^{s-m}\abs[1]{w}_{H^s(K)}, \quad \forall m \in \cbr[1]{0,\ldots,s},
				\label{eq:L2projerrors}
			\end{equation}
			where $C$ is independent of both $K$ and $h_K$. Moreover, if $s \geq 1$, then for all $K \in \mathcal{E}_h$ and all $e \in \partial K$, there holds
			\begin{equation}
				\norm[1]{w - \pi_{h,\Gamma} w}_{L^2(e)} \leq C h_K^{s - 1/2}\abs[1]{w}_{H^s(K)},
				\label{eq:L2projerrors_face}
			\end{equation}
			and if $s \geq 2$,
			\begin{equation}
				\norm[1]{\nabla\del[1]{w - \pi_{h,\Gamma}w}|_K\cdot\boldsymbol{n}_K}_{L^2(e)} \leq C h_K^{s - 3/2}\abs[1]{w}_{H^s(K)}.
				\label{eq:L2projerrors_gradface}
			\end{equation}
		\end{lemma}
		\noindent Note that these results imply that
		\begin{equation}
			\vertiii{w - \pi_{h,\Gamma} w}_{*} \leq Ch_K^{s-1}\abs[1]{w}_{H^s(\Omega)}.
			\label{eq:StabNorm_L2projerrors}
		\end{equation}
		The projected velocity $\bfu_h^{n+1}$ defined by \eqref{eq:RTproj1}-\eqref{eq:RTproj2}
		% $\Pi_\mathrm{RT}(-\kappa \overline{(\rho\lambda)}_t^n \nabla P_h^{n+1})$ 
		satisfies the following approximation bound. 
		\begin{lemma}
			Assume $p_\ell$ belongs to $L^2(0,T;H^{2}(\Omega))$. 
			There is a positive constant independent of $h$ and $\tau$ such that
			\begin{equation}
				\Vert \bfu_h^{n+1} + \kappa \nabla_h P_h^{n+1} \Vert \leq 
				C \vertiii{P_h^{n+1}-p_\ell^{n+1}} + C h.
				\label{eq:RT_estimate}
			\end{equation}
		\end{lemma}
		\begin{proof}	
			The proof of this bound follows an argument in
			\cite{BastianRiviere} and we present its main points.
			Let us denote
			\[
			\bfchi = \bfu_h^{n+1} + \kappa \nabla P_h^{n+1}.
			\]
			Then, from \eqref{eq:RTproj1}-\eqref{eq:RTproj2}, we have for any $K, K' \in \mathcal{E}_h$, and any $e\subset \partial K$,
			\begin{subequations}
				\begin{align*}
					\int_e \bfchi|_E\cdot\bfn_e q_h &= \frac12  \int_e \kappa(\nabla P_h^{n+1}|_E  
					- \nabla P_h^{n+1}|_{E'})  \cdot \bfn_e q_h
					+ \alpha_{p_\ell,e}h_e^{-1} \int_e \eta_{p_\ell,e}^n [P_h^{n+1}] q_h, \quad &e=\partial K \cap \partial K',
					\\
					\int_e \bfchi|_E\cdot\bfn_e q_h &= 0, \quad &e\subset\partial \Omega.
				\end{align*}
			\end{subequations}
			%\begin{eqnarray*}
			%\int_e \bfchi|_E\cdot\bfn_e q_h = \frac12  \int_e \kappa(\nabla P_h^{n+1}|_E  
			%- \nabla P_h^{n+1}|_{E'})  \cdot \bfn_e q_h
			%+ \alpha_{p_\ell,e}h_e^{-1} \int_e \eta_{p_\ell,e}^n [P_h^{n+1}] q_h, \quad e=\partial E \cap \partial E'\\
			%\int_e \bfchi|_E\cdot\bfn_e q_h = 0, \quad e\subset\partial \Omega
			%\end{eqnarray*}
			Let us take $q_h = \bfchi\cdot\bfn_e$ in the above; this is allowed because $P_h^{n+1}$ is piecewise linear and $\kappa$ is assumed to be piecewise constant 
			(see  \ref{hyp_kappa}).
			For edges on the boundary we have
			\[
			\Vert \bfchi|_E\cdot\bfn_e \Vert_{L^2(e)} = 0.
			\]
			For interior edges, we apply Cauchy-Schwarz's inequality:
			\[
			\Vert \bfchi|_E\cdot\bfn_e \Vert_{L^2(e)} \leq C \Vert  
			[\nabla P_h^{n+1}]\cdot\bfn_e\Vert_{L^2(e)}
			+ C h_e^{-1} \Vert [ P_h^{n+1}]\Vert_{L^2(e)}.
			\]
			We now bound $\Vert \bfchi\Vert_{L^2(K)}$ by passing to the reference element, by using the fact that
			$\Vert \cdot\Vert_{L^2(\partial \hat{K})}$ is a norm for the Raviart-Thomas space restricted to $\hat{K}$ and
			by going back to the physical element:
			\[
			\Vert \bfchi\Vert_{L^2(K)} \leq C h \Vert \hat{\bfchi}\Vert_{L^2(\hat{K})} \leq C h \Vert \hat{\bfchi}\Vert_{L^2(\partial\hat{K})}
			\leq  C h^{1/2} \Vert \bfchi\Vert_{L^2(\partial K)}.
			\]
			We apply the bounds above:
			\[
			\Vert \bfchi\Vert_{L^2(K)} \leq C h^{1/2} \sum_{e\in\partial E\setminus \partial\Omega} \Vert  
			[\nabla P_h^{n+1}]\cdot\bfn_e\Vert_{L^2(e)}
			+ C \sum_{e\in\partial E\setminus \partial\Omega} h^{-1/2} \Vert [ P_h^{n+1}]\Vert_{L^2(e)}.
			\]
			Taking the square and summing over all the elements:
			\[
			\Vert \bfchi\Vert^2 \leq C h \sum_{K \in \mathcal{E}_h} \del[2]{ \sum_{e\in\partial K\setminus \partial\Omega} \Vert  
				[\nabla P_h^{n+1}]\cdot\bfn_e\Vert_{L^2(e)}}^2
			+ C \sum_{K \in \mathcal{E}_h} \del[2]{\sum_{e\in\partial K\setminus \partial\Omega} h^{-1/2} \Vert [ P_h^{n+1}]\Vert_{L^2(e)}}^2.
			\]
			The last term is bounded above by $\vertiii{P_h^{n+1}-p_\ell^{n+1}}^2$ since $[p_\ell]=0$.
			For the first term, we write for $e=\partial K\cap\partial K'$
			\[
			\Vert [\nabla P_h^{n+1}]\Vert_{L^2(e)} \leq
			\Vert  [\nabla (P_h^{n+1}-p_\ell^{n+1})]\Vert_{L^2(e)}.
			\]
			Clearly, we have
			\[
			\Vert  [\nabla (P_h^{n+1}-p_\ell^{n+1})]\Vert_{L^2(e)}
			\leq C (\Vert \nabla (P_h^{n+1}-p_\ell^{n+1})|_K \Vert_{L^2(e)}
			+ \Vert \nabla (P_h^{n+1}-p_\ell^{n+1})|_{K'} \Vert_{L^2(e)}).
			\]  
			We add and subtract the $L^2$ projection of $p_\ell^{n+1}$ onto $X_h$:
			\[
			\Vert \nabla (P_h^{n+1}-p_\ell^{n+1})|_K \Vert_{L^2(e)}
			\leq \Vert \nabla (P_h^{n+1}-\pi_{h,\Gamma_D^{p_\ell}}p_\ell^{n+1})|_K \Vert_{L^2(e)}
			+ \Vert \nabla (\pi_{h,\Gamma_D^{p_\ell}}p_\ell^{n+1} - p_\ell^{n+1})|_K \Vert_{L^2(e)}
			\]
			\[
			\leq C h^{-1/2} \Vert \nabla (P_h^{n+1} - \pi_{h,\Gamma_D^{p_\ell}}p_\ell^{n+1})\Vert_{L^2(K)}
			+ C h^{1/2} \Vert p_\ell^{n+1}\Vert_{H^2(K)}.
			\]
			So
			\[
			h \sum_{K \in \mathcal{E}_h} \del[2]{\sum_{e\in\partial K\backslash \partial\Omega} \Vert [ \nabla (P_h^{n+1}-p_\ell^{n+1})\Vert_{L^2(e)}}^2
			\leq C \sum_{K \in \mathcal{E}_h} \Vert \nabla (P_h^{n+1} - \pi_{h,\Gamma_D^{p_\ell}}p_\ell^{n+1})\Vert_{L^2(K)}^2
			+ C h^2 \Vert p_\ell^{n+1}\Vert_{H^2(\Omega)}^2,
			\]
			or 
			\[
			\leq C \vertiii{P_h^{n+1}-p_\ell^{n+1}}^2 + C h^2 \Vert p_\ell^{n+1}\Vert_{H^2(\Omega)}^2.
			\]
			Combining all the bounds we have
			\[
			\Vert \bfchi\Vert \leq C \vertiii{P_h^{n+1}-p_\ell^{n+1}}
			+ C h.
			\]
		\end{proof}
		
		%---------------------------------------------
		\subsection{Hypotheses}
		%---------------------------------------------
		
		In the remaining of the paper, the following assumptions are made on the input data.
		\begin{enumerate}[label=\textbf{H.\arabic*}]
			\item \label{hyp_bounds} The nonlinear functions $\lambda_i$, for $i=v,\ell,a$, are $C^2$ functions with respect to time. Moreover, we have the following bounds:
			%The porosity $\phi$, phase mobilities $\lambda_i$, phase densities $\rho_i$, $i=\ell,v,a$, the absolute permeability $\kappa$, $\del[1]{\partial_{s_a}p_{c,a}}$, and the function $\psi_\ell$ are bounded above and below by positive constants:
			%
			\begin{equation}
				\begin{array}{ccccccc}
					%			0 &<& \phi_\ast &\leq & \phi 
					%			&\leq& \phi^\ast,
					%			\\
					0 &<& \underline{C}_{\del[0]{\rho\lambda}_t} &\leq & \del[1]{\rho\lambda}_t
					&\leq& \overline{C}_{\del[0]{\rho\lambda}_t},
					%			\\
					%			0 &<& \rho_{i_\ast}
					%			&\leq & \rho_i
					%			&\leq & \rho_i^\ast,
					\\
					0 &<& \underline{C}_{\lambda_i} 
					&\leq& \lambda_i 
					&\leq& \overline{C}_{\lambda_i},
					\\
					0 &<& \underline{C}_{\lambda_t} 
					&\leq& \lambda_t
					&\leq& \overline{C}_{\lambda_t},
					\\
					0 &<& \kappa_* 
					&\leq& \kappa 
					&\leq& \kappa^*,
					\\
					0 &<& \underline{C}_{p_{c,a}} 
					&\leq& \del[1]{\partial_{s_a}p_{c,a}}^{+} 
					&\leq& \overline{C}_{p_{c,a}},
					\\
					0 &\leq & \underline{C}_{p_{c,v}} 
					&\leq&  \partial_{s_v}p_{c,v}   
					&\leq& \overline{C}_{p_{c,v}}.
				\end{array}
			\end{equation}
			\begin{remark}
				We note that the above bounds also hold when these functions are evaluated with discrete solutions by using cutoff in the definition of the above functions.
			\end{remark}
			
			\item \label{hyp_lipschitz} The following functions are Lipschitz continuous, so that we have
			\begin{equation}
				\begin{array}{ccc}
					\abs[1]{\lambda_i(s_{a_1}, s_{v_1}) - \lambda_i(s_{a_2}, s_{v_2})} &\leq& L\del{\abs[1]{s_{a_1} - s_{a_2}} + \abs[1]{s_{v_1} - s_{v_2}}},
					\\
					\abs[1]{\partial_{s_v} p_{c,v}(s_{v_1}) - \partial_{s_v} p_{c,v}(s_{v_2})} & \leq&  L \abs[1]{s_{v_1} - s_{v_2}},  
					\\
					\abs[1]{\partial_{s_v} p_{c,a}(s_{a_1}) - \partial_{s_v} p_{c,a}(s_{a_2})} &\leq& L \abs[1]{s_{a_1} - s_{a_2}}.
				\end{array}
			\end{equation}
			
			\item \label{hyp_pcgradients} The functions $\nabla p_{c,a}$ and $\nabla p_{c,v}$ are bounded, so that we have 
			\begin{equation}
				\begin{array}{ccccccc}
					0 \leq \underline{C}_{\nabla p_{c_a}} \leq \norm[1]{\nabla p_{c_a}}_{L^\infty(\Omega)} &\leq& \overline{C}_{\nabla p_{c_a}},
					\\
					0 \leq \underline{C}_{\nabla p_{c_v}} \leq \norm[1]{\nabla p_{c_v}}_{L^\infty(\Omega)} &\leq& \overline{C}_{\nabla p_{c_v}},
				\end{array}
			\end{equation}
			and they satisfy the following growth conditions
			\begin{equation}
				\begin{array}{ccc}
					\norm[1]{\nabla p_{c_a}(s_{a_1}) - \nabla p_{c_a}(s_{a_2})} &\leq& L\norm[1]{s_{a_1} - s_{a_2}},
					\\
					\norm[1]{\nabla p_{c_v}(s_{v_1}) - \nabla p_{c_v}(s_{v_2})} &\leq& L\norm[1]{s_{v_1} - s_{v_2}}.
				\end{array}
			\end{equation}
			We remark that, even though this hypothesis might be somewhat restrictive, it has been used before in e.g., \cite{chen2001error,radu2018robust}. For instance, in \cite{chen2001error}, assumptions (A5) and (A7) state that the functions $\gamma_1$ and $\gamma_2$, which contain the gradient of the capillary pressure, are bounded and Lipschitz continuous with respect to the primary unknown $\theta$.
			
			\item \label{hyp_boundedsource} The source terms  $q_i$ are smooth enough, $q_i \in L^\infty(0,T;L^\infty(\Omega))$,
			for $i=\ell,v,a$.
			% and
			%$\partial_t q_i \in L^\infty(0,T;L^2(\Omega))$, for $i=\ell,v,a$.

			\item \label{hyp_kappa} The absolute permeability $\kappa$ is piecewise constant.% belongs to $H^2(\Omega)$. \Rd we can remove this, we do not use it directly, but it will be implied if we assume enough smoothness for the solution ? \Bk
			
			%\item \label{hyp:phi_rhoa_C2L2} $\tilde{\phi}\tilde{\rho}_a \in C^2(L^2(\Omega))$.
			
			%\item \label{hyp:phi_rhov_C2L2} $\tilde{\phi}\tilde{\rho}_v \in C^2(L^2(\Omega))$.
			
			%\item \label{hyp:rho_lambdat_C1Linf} $\del[1]{\tilde{\rho}\tilde{\lambda}}_t \in C^1(L^\infty(\Omega))$.
		\end{enumerate}

		%=============================================
		\section{Existence and uniqueness}
		%=============================================
		\label{sec:existence_uniqueness}
		
		In the following, we denote by $p_\ell$, $s_a$ and $s_v$ the exact solutions to \cref{eq:liquid_pressure_eq}, \cref{eq:aqueous_saturation_eq} and \cref{eq:vapor_saturation_eq}. We assume that the exact solutions are smooth enough, more precisely $p_\ell, s_v, s_a\in C^2(0,T;L^{2}(\Omega)) \cap C^0(0,T;H^2(\Omega)) \cap L^\infty(0,T;W^{1,\infty}(\Omega))$.
		%$\partial_t p_\ell, \partial_t s_a, \partial_t s_v \in L^\infty(\Omega\times (0,T))$ and
		%$\partial_{tt} p_\ell, \partial_{tt} s_a, \partial_{tt} s_v\in L^\infty(0,T;L^2(\Omega))$.
		
		%\gsj{GSJ: to apply Taylor expansion, we may need unknown in $C^2(0,T; L^2)$ and $C^0(0,T,H^1)$}
		
		%\Rd ?belong to $C^0(0,T_f;H^2(\Omega)) \cap C^2(0,T_f;L^2(\Omega))$.?
		
		%\Rd BR: we need to check how to write the conditions on the nonlinear functions:\Bk
		%\begin{eqnarray*}
		%\tilde{\phi}\tilde{\rho}_t, \tilde{\phi}\tilde{\rho}_a, \tilde{\phi}\tilde{\rho}_v   \in C^2(0,T;L^2(\Omega))\\
		%\del[1]{\tilde{\rho}\tilde{\lambda}}_t \in C^1(0,T;L^\infty(\Omega))
		%\end{eqnarray*}
		%\Rd maybe we also need
		%\[
		%\tilde{\rho}_i \in \mathcal{C}^1(0,T;L^2(\Omega)) 
		%\]
		%\Bk

		For readability, we denote by $\tilde{\lambda}_t, \tilde{p}_{c,a}, \tilde{p}_{c,v}, \tilde{\lambda}_i$ for $i=v,\ell,a$ the functions $\phi, \lambda_t, p_{c,a}, p_{c,v}, \lambda_i$ evaluated at the exact solutions (pressures and saturations) at time $t$.  If the time is $t_n$, then the functions are denoted by $\tilde{\phi}^n, \tilde{\rho}_t^n, \tilde{p}_{c,a}^n, \tilde{p}_{c,v}^n, \tilde{\lambda}_i^n$. For instance, we will write:
		\[
		\tilde{\lambda}_a^n = \lambda_a(s_a^n, s_v^n), \quad \lambda_a^n = \lambda_a(S_{a_h}^n, S_{v_h}^n).
		\]
		
		Existence and uniqueness of $P_h^{n+1}, S_{a_h}^{n+1}, S_{v_h}^{n+1}$ follow from the linearity of \eqref{eq:disc_pl}, \eqref{eq:disc_Sa}, \eqref{eq:disc_Sv} with respect to their unknowns and from the coercivity and continuity of the forms $b_{p}$, $b_{a}$ and $b_{v}$.
		%---------------------------------------------
		\subsection{Liquid pressure}
		%---------------------------------------------
		
		\begin{lemma}[Consistency of $b_{p}$]
			We have  for any $n \geq 0$    and any $w_h \in X_{h,\Gamma_\D^{p_\ell}}$
			\begin{equation}
				\tilde{b}^{n+1}_{p}(p_\ell^{n+1}, w_h) = \tilde{f}^{n+1}_{p}(w_h), 
				\label{eq:consistency_pressure}
			\end{equation}
			where
			\begin{equation}
				\tilde{b}^{n+1}_{p}(p_\ell,w_h) = b_{p}(p_\ell, w_h; p_\ell^{n+1}, s_a^{n+1}, s_v^{n+1}), \quad \text{and} 
				\quad \tilde{f}^{n+1}_{p}(w_h) = f_{p}(w_h; p_\ell^{n+1}, s_a^{n+1}, s_v^{n+1}).
			\end{equation}
			\label{lem:consistency_pressure}
		\end{lemma}
		
		\begin{proof}
			First, note that
			\begin{equation}
				\tilde{b}^{n+1}_{p}(p_\ell,w_h) = \sum_{K \in \mathcal{E}_h}\int_K \tilde{\lambda}_t^{n+1}\kappa \nabla p_\ell\cdot \nabla w_h 
				-\sum_{e \in \Gamma_h }\int_e \tilde{\lambda}_t^{n+1} \kappa \nabla p_\ell\cdot \boldsymbol{n}_e \sbr[1]{w_h}.
			\end{equation}
			In the rest of the proof, we drop the superscript $(n+1)$ for readability, but it is understood that all functions are evaluated at time
			$t_{n+1}$. 
			Applying integration by parts on the first term, we obtain
			\begin{equation}
				\tilde{b}_{p}(p_\ell,w_h) =-\sum_{K \in \mathcal{E}_h}\int_K \nabla\cdot\del[1]{\tilde{\lambda}_t\kappa \nabla p_\ell} w_h 
				+ \sum_{K \in \mathcal{E}_h}\int_{\partial K}\tilde{\lambda}_t\kappa\nabla p_\ell \cdot \boldsymbol{n}_K w_h
				-\sum_{e \in \Gamma_h }\int_e \tilde{\lambda}_t \kappa \nabla p_\ell\cdot \boldsymbol{n}_e \sbr[1]{w_h}.
			\end{equation}
			%
			%Using the well-known identity for any scalar function $w$ and vector function $\boldsymbol{\tau}$
			%
			%\begin{equation}
			%\sum_{K \in \mathcal{E}_h}\int_{\partial K} w \boldsymbol{\tau}\cdot\boldsymbol{n}_K = \sum_{e \in \Gamma_h \cup \partial\Omega } \int_e \sbr[1]{w}\cbr[1]{\boldsymbol{\tau}\cdot\boldsymbol{n}_e} + \sum_{e \in \Gamma_h} \int_e \cbr[1]{w}\sbr[1]{\boldsymbol{\tau}\cdot\boldsymbol{n}_e},
			%\label{eq:magic_formula}
			%\end{equation}
			%
			Using the fact that $\sbr[1]{\tilde{\lambda}_t\kappa\nabla p_\ell\cdot \boldsymbol{n}_e} = 0$ on interior faces, we obtain
			%, and $\cbr[1]{\del[0]{\tilde{\rho}\tilde{\lambda}}_t^{n+1}\kappa\nabla p_\ell^{n+1}\cdot \boldsymbol{n}_e} = \del[0]{\tilde{\rho}\tilde{\lambda}}_t\nabla p_\ell\cdot \boldsymbol{n}_e\kappa$, we obtain
			%
			\begin{equation}
				\begin{split}
					\tilde{b}_{p}(p_\ell,w_h) =& -\sum_{K \in \mathcal{E}_h}\int_K \nabla\cdot\del[1]{\tilde{\lambda}_t\kappa \nabla p_\ell} w_h 
					+ \sum_{e \in \Gamma_{\text{N}}^{p_\ell}}\int_{e} \tilde{\lambda}_t \kappa\nabla p_\ell \cdot \boldsymbol{n}_e w_h.
					%\\
					%&+ \sum_{e \in \Gamma_{\text{D}}^{p_\ell}}\alpha_{p_\ell,e} h_e^{-1}\int_e \eta_{p_\ell,e} p_\ell w_h 
					%+ \theta_{p_\ell} \sum_{e \in \Gamma_{\text{D}}^{p_\ell}}\int_e \del[0]{\tilde{\rho}\tilde{\lambda}}_t \kappa\nabla w_h \cdot\boldsymbol{n}_e p_\ell.
				\end{split}
			\end{equation}
			On the other hand, after integration by parts on the volume term of \cref{eq:fh_pressure}, we have
			\begin{equation}
				\begin{split}
					\tilde{f}_{p}(w_h) &= (q_t, w_h)
					+ \sum_{K \in \mathcal{E}_h}\int_K \nabla\cdot\del[1]{ \tilde{\lambda}_v \kappa\nabla \tilde{p}_{c,v} - \tilde{\lambda}_a \kappa \nabla \tilde{p}_{c,a} - \kappa\del[0]{\rho \tilde{\lambda}}_t\boldsymbol{g} } w_h
					\\
					& - \sum_{K \in \mathcal{E}_h}\int_{\partial K} \del[1]{\tilde{\lambda}_v \kappa\nabla \tilde{p}_{c,v} - \tilde{\lambda}_a \kappa\nabla \tilde{p}_{c,a} - \kappa\del[0]{\rho \tilde{\lambda}}_t \boldsymbol{g} }\cdot \boldsymbol{n}_K w_h
					\\
					& + \sum_{e \in \Gamma_{\text{N}}^{p_\ell}}\int_e j_p^{\text{N}} w_h
					+\sum_{e \in \Gamma_h }\int_e \cbr[1]{\tilde{\lambda}_v \kappa\nabla \tilde{p}_{c,v}\cdot \boldsymbol{n}_e}\sbr[1]{w_h} 
					- \sum_{e \in \Gamma_h }\int_e \cbr[1]{\tilde{\lambda}_a \kappa\nabla \tilde{p}_{c,a} \cdot \boldsymbol{n}_e}\sbr[1]{w_h}
					\\
					&- \sum_{e \in \Gamma_h }\int_e \cbr[1]{\kappa\del[0]{\rho \tilde{\lambda}}_t\boldsymbol{g} \cdot \boldsymbol{n}_e}\sbr[1]{w_h}.
				\end{split}
			\end{equation}
			Using that $\sbr[1]{\del[1]{\tilde{\lambda}_v \kappa\nabla \tilde{p}_{c,v} - \tilde{\lambda}_a \kappa \nabla \tilde{p}_{c,a} - \kappa\del[0]{\rho \tilde{\lambda}}_t \boldsymbol{g}}\cdot \boldsymbol{n}} = 0$ on interior faces, we obtain
			\begin{multline}
				\tilde{f}_{p}(w_h) = (q_t, w_h)
				+ \sum_{K \in \mathcal{E}_h}\int_K 	\nabla\cdot\del[1]{\tilde{\lambda}_v \kappa\nabla \tilde{p}_{c,v} - \tilde{\lambda}_a \kappa\nabla \tilde{p}_{c,a} - \kappa\del[0]{\rho \tilde{\lambda}}_t\boldsymbol{g} } w_h
				\\
				- \sum_{e \in \Gamma_{\text{N}}^{p_\ell}}\int_e \tilde{\lambda}_v \kappa\nabla \tilde{p}_{c,v} \cdot \boldsymbol{n}_e w_h 
				+ \sum_{e \in \Gamma_{\text{N}}^{p_\ell}}\int_e \tilde{\lambda}_a \kappa\nabla \tilde{p}_{c,a} \cdot \boldsymbol{n}_e w_h 
				+ \sum_{e \in \Gamma_{\text{N}}^{p_\ell}}\int_e \kappa\del[0]{\rho \tilde{\lambda}}_t \boldsymbol{g} \cdot \boldsymbol{n}_e w_h
				+ \sum_{e \in \Gamma_{\text{N}}^{p_\ell}}\int_e j_p^{\text{N}} w_h.
			\end{multline}
			Recalling that $p_\ell$ solves \cref{eq:liquid_pressure_eq} and satisfies the boundary condition \cref{eq:Neumann_BC_pressure}, the result \cref{eq:consistency_pressure} follows.
		\end{proof}
		
		\begin{lemma}
			For all $\del[1]{v_h, w_h} \in X_h \times X_h$, the following relation is satisfied for all $n\geq 0$
			\begin{equation}
				\abs{\sum_{e \in \Gamma_h } \int_e \cbr[1]{\lambda_t^n \kappa\nabla v_h \cdot \boldsymbol{n}_e}\sbr[1]{w_h}} 
				\leq 
				\overline{C}_{\lambda_t} \kappa^* \del{\sum_{K \in \mathcal{E}_h} \sum_{e \in \partial K} h_e \norm[1]{\nabla v_h|_K \cdot \boldsymbol{n}_e}_{L^2(e)}^2 }^{1/2}\abs[1]{w_h}_{\text{J}}.
				\label{eq:bound_consistency_pressure}
			\end{equation}
		\end{lemma}
		
		\begin{proof}
			%First, let $e \in \Gamma_{\text{D}}^{p_\ell}$, then, using Cauchy--Schwarz's inequality we have 
			%
			%\begin{equation}
			%	\begin{split}
				%		\int_e \cbr[1]{\del[0]{\rho\lambda}_t^n \kappa \nabla P_h \cdot \boldsymbol{n}_e}\sbr[1]{w_h} &= \int_e \del[0]{\rho\lambda}_t^n \kappa\nabla P_h \cdot \boldsymbol{n}_e w_h 
				%		\\
				%		& \leq \overline{C}_{\del[0]{\rho\lambda}_t} \kappa^* h_e^{1/2} \norm[1]{\nabla P_h \cdot \boldsymbol{n}_e}_{L^2(e)} h_e^{-1/2}\norm[1]{w_h}_{L^2(e)}.
				%	\end{split}
			%	\label{eq:ineq_bdr_face}
			%\end{equation}
			%
			Let us consider a face $e \in \Gamma_h$ that is shared between two elements $K_1$ and $K_2$, i.e., $e = \partial K_1 \cap \partial K_2$. With \ref{hyp_bounds} and Cauchy-Schwarz's inequality, we have
			\begin{equation}
				\int_e \cbr[1]{\lambda_t^n \kappa \nabla v_h \cdot \boldsymbol{n}_e}\sbr[1]{w_h} 
				\leq \overline{C}_{\lambda_t} \kappa^* h_e^{1/2} \del{\norm[1]{\nabla v_h|_{K_1} \cdot \boldsymbol{n}_e }_{L^2(e)}
					+ \norm[1]{\nabla v_h|_{K_2} \cdot \boldsymbol{n}_e}_{L^2(e)}}h_e^{-1/2}\norm[1]{\sbr[0]{w_h}}_{L^2(e)}.
			\end{equation}
			Summing over all faces, applying  Cauchy--Schwarz's inequality and writing the sum in terms of the face contributions for each element, we obtain the result.
		\end{proof}
		
		Next, we show that $b_{p}$ is coercive on $X_{h,\Gamma_\D^{p_\ell}}$.
		
		\begin{lemma}[Coercivity of $b_{p}$]
			Assume that $\alpha_{p_\ell,*}$ satisfies that
			\begin{equation}
				\alpha_{p_\ell,*} > 0.25\del{1 - \theta_{p_\ell}}^2 \del[1]{\overline{C}_{\lambda_t} \kappa^*}^3\del[1] {\underline{C}_{\lambda_t} \kappa_*}^{-3} C_{\text{tr}}^2,
			\end{equation}
			where $C_{\text{tr}}$ results from the trace inequality \cref{eq:trace_ineq}. Then, the bilinear form $b_{p}^n$ defined by \cref{eq:ah_pressure} is coercive on $X_h$ with respect to the norm $\vertiii[1]{\cdot}$ defined by \cref{eq:dif_norm}, i.e., for all $w_h \in X_{h,\Gamma_\D^{p_\ell}}$ and for all $n\geq 0$, the following relation is satisfied:
			\begin{equation}
				b_{p}^n(w_h,w_h) \geq C_{\alpha,p_\ell}\vertiii[1]{w_h}^2,
				\label{eq:coercivity_pressure}
			\end{equation}
			with
			%
			%		\begin{equation}
				%			C_{\alpha,p_\ell} = \frac{0.25\del{1 - \theta_{p_\ell}}^2 \del[1]{\overline{C}_{\del[0]{\rho\lambda}_t} \kappa^*}^2\del[1] {\underline{C}_{\del[0]{\rho\lambda}_t} \kappa_*}^{-2} C_{\text{tr}}^2}{1 + \alpha_{p_\ell,*}\underline{C}_{\del[0]{\rho\lambda}_t} \kappa_*\del[1]{\overline{C}_{\del[0]{\rho\lambda}_t} \kappa^*}^{-1}}\underline{C}_{\del[0]{\rho\lambda}_t} \kappa_*.
				%		\end{equation}
			\begin{equation}
				C_{\alpha,p_\ell} = \frac{\alpha_{p_\ell,*}\underline{C}_{\lambda_t} \kappa_*\del[1]{\overline{C}_{\lambda_t} \kappa^*}^{-1} - 0.25\del{1 - \theta_{p_\ell}}^2 		\del[1]{\overline{C}_{\lambda_t} \kappa^*}^2\del[1] {\underline{C}_{\lambda_t} \kappa_*}^{-2} C_{\text{tr}}^2}{1 + \alpha_{p_\ell,*}\underline{C}_{\lambda_t} \kappa_*\del[1]{\overline{C}_{\lambda_t} \kappa^*}^{-1}}.
			\end{equation}
			\label{lem:coercivity_pressure}
		\end{lemma}
		
		\begin{proof}
			Using \cref{eq:ah_pressure}, we have
			\begin{equation}
				\begin{split}
					b_{p}^n(w_h,w_h) &= \sum_{K \in \mathcal{E}_h}\int_K \lambda_t^n \kappa \abs[1]{\nabla w_h}^2
					+ \sum_{e \in \Gamma_h }\alpha_{p_\ell,e}h_e^{-1}\int_e \eta_{p_\ell,e}^n [w_h]^2  
					+(\theta_{p_\ell} - 1) \sum_{e \in \Gamma_h }\int_e \cbr[1]{\lambda_t^n \kappa\nabla w_h\cdot \boldsymbol{n}_e}\sbr[1]{w_h}
					\\
					& \geq \underline{C}_{\lambda_t} \kappa_* \norm[1]{\nabla w_h}^2 + \alpha_{p_\ell,*} \frac{\underline{C}_{\lambda_t}^2 \kappa_*^2}{\overline{C}_{\lambda_t} \kappa^*}\abs[1]{w_h}_{\text{J}}^2
					+ \del[1]{\theta_{p_\ell} - 1}\sum_{e \in \Gamma_h }\int_e \cbr[1]{\lambda_t^n \kappa \nabla w_h\cdot \boldsymbol{n}_e}\sbr[1]{w_h}.
				\end{split}
				\label{eq:ah_w_w}
			\end{equation}
			Using \cref{eq:bound_consistency_pressure}, the trace inequality \cref{eq:trace_ineq} and the fact that for all $K \in \mathcal{E}_h$ and all $e \in \partial K$, $h_e \leq h_K$, we have
			\begin{equation}
				\begin{split}
					\abs{\sum_{e \in \Gamma_h } \int_e \cbr[1]{\lambda_t^n \kappa \nabla w_h \cdot \boldsymbol{n}_e}\sbr[1]{w_h}} & \leq \overline{C}_{\lambda_t} \kappa^* \del{\sum_{K \in \mathcal{E}_h} \sum_{e \in \partial K} h_e \norm[1]{\nabla w_h|_K \cdot \boldsymbol{n}_e}_{L^2(e)}^2 }^{1/2}\abs[1]{w_h}_{\text{J}}
					\\
					&\leq \overline{C}_{\lambda_t} \kappa^* C_{\text{tr}} \norm[1]{\nabla w_h} \abs[1]{w_h}_{\text{J}}.
				\end{split}
			\end{equation}
			Thus, since $\theta_{p_\ell} -1 \leq 0$, we have
			\begin{equation}
				\del[1]{\theta_{p_\ell} - 1}\sum_{e \in \Gamma_h }\int_e \cbr[1]{\lambda_t^n \kappa \nabla w_h\cdot \boldsymbol{n}_e}\sbr[1]{w_h} 
				\geq \del[1]{\theta_{p_\ell} - 1}\overline{C}_{\lambda_t} \kappa^* C_{\text{tr}} \norm[1]{\nabla w_h} \abs[1]{w_h}_{\text{J}}.
			\end{equation}
			Using this in \cref{eq:ah_w_w} and noting that
			$\theta_{p_\ell} - 1$ is equal to either -2, -1 or 0,  we have 
			%
			%\begin{equation}
			%\begin{split}
			%b_{p}^n(w_h,w_h) \geq & \underline{C}_{\lambda_t} \kappa_* \norm[1]{\nabla w_h}^2 + \alpha_{p_\ell,*} \frac{\underline{C}_{\lambda_t}^2 \kappa_*^2}{\overline{C}_{\lambda_t} \kappa^*} \abs[1]{w_h}_{\text{J}}^2
			%\\
			%&+ \del[1]{\theta_{p_\ell} - 1}\overline{C}_{\lambda_t} \kappa^* C_{\text{tr}} \norm[1]{\nabla w_h} \abs[1]{w_h}_{\text{J}}.
			%\end{split}
			%\label{eq:ah_w_thetapl}
			%\end{equation}
			%%
			%Note that $\theta_{p_\ell} - 1$ is equal to either -2, -1 or 0; therefore, we can write \cref{eq:ah_w_thetapl} as
			%
			\begin{equation}
				%\begin{split}
				b_{p}^n(w_h,w_h) \geq \underline{C}_{\lambda_t} \kappa_* \norm[1]{\nabla w_h}^2 + \alpha_{p_\ell,*} \frac{\underline{C}_{\lambda_t}^2 \kappa_*^2}{\overline{C}_{\lambda_t} \kappa^*} \abs[1]{w_h}_{\text{J}}^2
				%\\
				-2 \frac{1 - \theta_{p_\ell}}{2} \overline{C}_{\lambda_t} \kappa^* C_{\text{tr}} 	\norm[1]{\nabla w_h} \abs[1]{w_h}_{\text{J}}.
				%\end{split}
				\label{eq:ah_w_thetapl2}
			\end{equation}	
			Next, we use the following inequality: let $\beta$ be a nonnegative real number and assume that $c > \beta^2$, then, for all $x,y \in \mathbb{R}$,
			\begin{equation}
				x^2 - 2\beta x y + c y^2 \geq \frac{c - \beta^2}{1 + c}\del[1]{x^2 + y^2}.
				\label{eq:important_ineq}
			\end{equation}
			Using this in \cref{eq:ah_w_thetapl2} with $c = \alpha_{p_\ell,*} \underline{C}_{\lambda_t} \kappa_*\del[1]{\overline{C}_{\lambda_t} \kappa^*}^{-1}$, $\beta = 0.5\del{1 - \theta_{p_\ell}}C_{\text{tr}}\overline{C}_{\lambda_t} \kappa^* \del[1]{\underline{C}_{\lambda_t} \kappa_*}^{-1}$, $x = \del[1]{\underline{C}_{\lambda_t} \kappa_*}^{1/2}\norm[1]{\nabla w_h}$ and $y =  \del[1]{\underline{C}_{\lambda_t} \kappa_*}^{1/2}\abs[1]{w_h}_{\text{J}}$, concludes the proof.
		\end{proof}
		
		Now, we prove that $b_{p}$ is bounded.
		
		\begin{lemma}[Boundedness of $b_{p}$]
			There exists a constant $C_{\text{B},p_\ell} > 0$ independent of $h$ such that, for all $v_h \in X_{h,\Gamma_\D^{p_\ell}}$ and $w_h \in X_{h,\Gamma_\D^{p_\ell}}$, the following relation is satisfied
			\begin{equation}
				b_{p}^n(v_h,w_h) \leq C_{\text{B},p_\ell}\vertiii{v_h}\,\, \vertiii{w_h}.
				\label{eq:boundedness_pressure}
			\end{equation}
			In addition, there exists a constant $C_{\text{B}_*,p_\ell} > 0$ independent of $h$ and $\tau$ such that for any  $v \in H^2(\Omega)+X_{h,\Gamma_\D^{p_\ell}}$ and any
			$w_h\in X_{h,\Gamma_\D^{p_\ell}}$, the following bound holds
			\begin{equation}
				\abs[1]{b_{p}^n(v,w_h)} \leq C_{\text{B}_*,p_\ell} \vertiii{v}_{\ast}\vertiii{w_h}.
				\label{eq:boundedness_continuous_pressure}
			\end{equation}
			\label{lem:boundedness_pressure}
		\end{lemma}
		
		\begin{proof}
			Let $v_h \in X_{h,\Gamma_\D^{p_\ell}}$ and $w_h \in X_{h,\Gamma_\D^{p_\ell}}$. We have
			\begin{equation}
				\begin{split}
					\abs[1]{b_{p}^n(v_h, w_h)} &\leq \abs{\sum_{K \in \mathcal{E}_h}\int_K \lambda_t^n\kappa \nabla v_h \cdot \nabla w_h}
					+ \abs{\sum_{e \in \Gamma_h }\alpha_{p_\ell,e}h_e^{-1}\int_e \eta_{p_\ell,e}^n \sbr[1]{v_h}\sbr[1]{w_h}}
					\\
					&+ \abs{\sum_{e \in \Gamma_h }\int_e \cbr[1]{\lambda_t^n \kappa\nabla v_h \cdot \boldsymbol{n}_e}\sbr[1]{w_h}}
					+ \abs{\sum_{e \in \Gamma_h }\theta_{p_\ell}\int_e \cbr[1]{\lambda_t^n \kappa\nabla w_h \cdot\boldsymbol{n}_e}\sbr[1]{v_h}}
					\\ &= T_1 + T_2 + T_3 + T_4.
				\end{split}
			\end{equation}
			Using  Cauchy--Schwarz's inequality, we see that
			\begin{equation}
				T_1 \leq \overline{C}_{\lambda_t} \kappa^* \norm[1]{\nabla_h v_h} \norm[1]{\nabla_h w_h} \leq \overline{C}_{\lambda_t} \kappa^* \vertiii{v_h} \,\, \vertiii{w_h}.
			\end{equation}
			Similarly,
			\begin{equation}
				T_2 \leq \alpha_{p_\ell}^* \frac{\del[1]{\overline{C}_{\lambda_t} \kappa^*}^2}{\underline{C}_{\lambda_t} \kappa_*} \abs[1]{v_h}_{\text{J}} \abs[1]{w_h}_{\text{J}} \leq \alpha_{p_\ell}^* \frac{\del[1]{\overline{C}_{\lambda_t} \kappa^*}^2}{\underline{C}_{\lambda_t} \kappa_*} \vertiii{v_h}\,\, \vertiii{w_h}.
			\end{equation}
			Using \cref{eq:bound_consistency_pressure}, recalling that $h_e \leq h_K$, and using the trace inequality \cref{eq:trace_ineq}, we bound $T_3$ as
			\begin{equation}
				T_3 \leq \overline{C}_{\lambda_t} \kappa^* \vertiii{v_h}\,\, \vertiii{w_h}.
			\end{equation}
			Finally, $T_4$ can be bounded in a similar way as $T_3$ to obtain
			\begin{equation}
				T_4 \leq \abs{\theta_{p_\ell}}\overline{C}_{\lambda_t} \kappa^* \vertiii{v_h}\, \vertiii{w_h}.
			\end{equation}
			Taking $C_{\text{B},p_\ell} = 4\overline{C}_{\lambda_t} \sqrt{\kappa^*} \max\left(\alpha_{p_\ell}^* \frac{\overline{C}_{\lambda_t} \kappa^*}{\underline{C}_{\lambda_t} \kappa_*}, \abs[1]{\theta_{p_\ell}}\right)$ gives the result.
			The proof of \eqref{eq:boundedness_continuous_pressure} is similar; one needs to change the bound for the term $T_3$.
		\end{proof}

		\begin{corollary}
			There exists a unique solution to problem \cref{eq:disc_pl}.
		\end{corollary}
		
		\begin{proof}
			The coercivity \cref{lem:coercivity_pressure} and boundedness \cref{lem:boundedness_pressure} of $b_{p}$, together with the fact that $\psi_\ell$ is strictly positive, imply, using the Lax--Milgram theorem, that problem \cref{eq:disc_pl} is well-posed.
		\end{proof}
		
		%---------------------------------------------
		\subsection{Aqueous saturation}
		%---------------------------------------------
		
		%	We define the function $f_a^\ast(w_h;p_\ell^{n+1},s_a^{n+1},s_v^{n+1})$ as follows:
		%	\begin{equation}
			%		\begin{split}
				%			f_a^\ast(w_h;&p_\ell^{n+1}, s_{a}^{n+1}, s_{v}^{n+1}) = (\tilde{\rho}_a^{n+1} q_a^{n+1}, w_h) +
				%			\sum_{K \in \mathcal{E}_h}\int_K \tilde{\rho}_a^{n+1} \kappa \tilde{\lambda}_a^{n+1} \del[1]{\partial_{s_v} \tilde{p}_{c,a}}^{n+1} \nabla s_{v}^{n+1} \cdot \nabla w_h 
				%			\\
				%			&+ \sum_{K \in \mathcal{E}_h}\int_K \del{\tilde{\rho}_a^{n+1} \tilde{\lambda}_a^{n+1}(-\kappa \nabla p_\ell^{n+1})  + (\tilde{\rho}_a^{n+1})^2 \tilde{\lambda}_a^{n+1} \boldsymbol{g}}\cdot \nabla w_h 
				%			\\
				%			&  +  \sum_{e \in \Gamma_{\text{N}}^{s_a}}\int_e j_{s_a}^{\text{N}} w_h 
				%			- \sum_{e \in \Gamma_h }\int_e \tilde{\rho}_a^{n+1}\kappa\tilde{\lambda}_a^{n+1} \del[1]{\partial_{s_v} \tilde{p}_{c,a}}^{n+1} \nabla s_{v}^{n+1}  \cdot \boldsymbol{n}_e \sbr[1]{w_h}
				%			\\
				%			&- \sum_{e \in \Gamma_h }\int_e \tilde{\rho}_a^{n+1} \tilde{\lambda}_a^{n+1} (-\kappa \nabla p_\ell^{n+1})  \cdot\boldsymbol{n}_e  \sbr[1]{w_h} 
				%			- \sum_{e \in \Gamma_h }\int_e  (\tilde{\rho}_a^{n+1})^2\kappa \tilde{\lambda}_a^{n+1}\boldsymbol{g}\cdot\boldsymbol{n}_e \sbr[1]{w_h}.%\\
				%		\end{split}
			%	\end{equation}
		%	We note that the function $f_a^\ast$ is very similar to the form $f_a$. 
		
		\begin{lemma}[Consistency of $b_{a}$]
			\begin{equation}
				(\phi\del[0]{\partial_ts_a}^{n+1},w_h) + \tilde{b}_{a}^{n+1}(s_a^{n+1}, w_h) = \tilde{f}^{n+1}_{a}(w_h), \quad \forall w_h \in X_{h,\Gamma_\D^{s_a}}, \quad \forall n\geq 0,
				\label{eq:consistency_saturation}
			\end{equation}
			where $\tilde{b}^{n+1}_{a}(s_a,w_h) = b_{a}(s_a,w_h;p_\ell^{n+1}, s_a^{n+1}, s_v^{n+1})$, and $\tilde{f}^{n+1}_{a}(w_h) = f_{a}(w_h; p_\ell^{n+1}, s_a^{n+1}, s_v^{n+1})$.
		\end{lemma}
		
		\begin{proof}
			The proof of this lemma is skipped because it is  analogous to the proof of \cref{lem:consistency_pressure}. 
		\end{proof}
		
		\begin{lemma}
			For all $\del[1]{v_{h}, w_h} \in X_h \times X_h$, the following relation is satisfied
			\begin{multline}
				\abs{\sum_{e \in \Gamma_h } \int_e \cbr{\kappa \lambda_a^n \del[1]{\partial_{s_a}p_{c,a}}^{+,n} \nabla v_{h} \cdot \boldsymbol{n}_e}\sbr[1]{w_h}} 
				\\
				\leq \overline{C}_{\lambda_a}\kappa^* \overline{C}_{p_{c,a}} \del{\sum_{K \in \mathcal{E}_h} \sum_{e \in \partial K} h_e \norm[1]{\nabla v_h|_K \cdot \boldsymbol{n}_e}_{L^2(e)}^2 }^{1/2}\abs[1]{w_h}_{\text{J}}.
				\label{eq:bound_consistency_aqueous_saturation}
			\end{multline}
		\end{lemma}
		
		\begin{proof}
			The proof of this lemma is analogous to the proof of \cref{eq:bound_consistency_pressure}.
		\end{proof}
		
		Now we can show that $b_{a}$ is coercive on $X_{h,\Gamma_\D^{s_a}}$.
		
		\begin{lemma}[Coercivity of $b_{a}$]
			Assume that $\alpha_{s_a,*}$ satisfies that
			\begin{equation}
				\alpha_{s_a,*} > 0.25\del{1 - \theta_{s_a}}^2 \del[1]{\overline{C}_{\lambda_a} \overline{C}_{p_{c,a}} \kappa^*}^3 \del[1]{\underline{C}_{\lambda_a} \underline{C}_{p_{c,a}} \kappa_*}^{-3} C_{\text{tr}}^2,
			\end{equation}
			where $C_{\text{tr}}$ results from the trace inequality \cref{eq:trace_ineq}. Then, the bilinear form $b_{a}^{n}$ defined by \cref{eq:ah_aqueous_saturation} is coercive on $X_{h,\Gamma_\D^{s_a}}$ with respect to the norm $\vertiii[1]{\cdot}$ defined by \cref{eq:dif_norm}, i.e., for all $w_h \in X_{h,\Gamma_\D^{s_a}}$, the following relation is satisfied:
			\begin{equation}
				b_{a}^{n}(w_h,w_h) \geq C_{\alpha,s_a}\vertiii[1]{w_h}^2,
				\label{eq:coercivity_aqueous_saturation}
			\end{equation}
			with
			\begin{equation}
				C_{\alpha,s_a} = \frac{\alpha_{s_a,*}\underline{C}_{p_{c,a}}\underline{C}_{\lambda_a} \kappa_*\del[1]{\overline{C}_{p_{c,a}} \overline{C}_{\lambda_a} \kappa^*}^{-1} - 0.25\del{1 - \theta_{s_a}}^2 		\del[1]{\overline{C}_{p_{c,a}} \overline{C}_{\lambda_a} \kappa^*}^2\del[1] {\underline{C}_{p_{c,a}}\underline{C}_{\lambda_a} \kappa_*}^{-2} C_{\text{tr}}^2}{1 + \alpha_{s_a,*}\underline{C}_{p_{c,a}} \underline{C}_{\lambda_a} \kappa_*\del[1]{\overline{C}_{p_{c,a}} \overline{C}_{\lambda_a} \kappa^*}^{-1}}.
				%			
				%			C_{\alpha,s_a} = \del[2]{1 + \alpha_{s_a,*}\underline{C}_{p_{c,a}} \underline{C}_{\lambda_a} \underline{C}_{\rho_a} \kappa_*\del[1]{\overline{C}_{p_{c,a}} \overline{C}_{\lambda_a} \overline{C}_{\rho_a} \kappa^*}^{-1}}^{-1} \del[2]{\underline{C}_{p_{c,a}} \underline{C}_{\rho_a} \kappa_* \underline{C}_{\lambda_a}}
				%			\\
				%			\del[2]{0.25\del{1 - \theta_{s_a}}^2 \del[1]{\overline{C}_{\rho_a}\overline{C}_{\lambda_a} \overline{C}_{p_{c,a}} \kappa^*}^2 \del[2]{\underline{C}_{\rho_a} \underline{C}_{\lambda_a} \underline{C}_{p_{c,a}} \kappa_*}^{-2} C_{\text{tr}}^2}.
			\end{equation}
			%		\begin{equation}
				%			C_{\alpha,s_a} = \frac{\alpha_{s_a,*}\underline{C}_{p_{c,a}} \underline{C}_{\lambda_a} \underline{C}_{\rho_a} \kappa_*\del[1]{\overline{C}_{p_{c,a}} \overline{C}_{\lambda_a} \overline{C}_{\rho_a} \kappa^*}^{-1} - 0.25\del{1 - \theta_{s_a}}^2 \del[1]{\overline{C}_{\rho_a}\overline{C}_{\lambda_a} \overline{C}_{p_{c,a}} \kappa^*}^2 \del[1]{\underline{C}_{\rho_a} \underline{C}_{\lambda_a} \underline{C}_{p_{c,a}} \kappa_*}^{-2} C_{\text{tr}}^2}{1 + \alpha_{s_a,*}\underline{C}_{p_{c,a}} \underline{C}_{\lambda_a} \underline{C}_{\rho_a} \kappa_*\del[1]{\overline{C}_{p_{c,a}} \overline{C}_{\lambda_a} \overline{C}_{\rho_a} \kappa^*}^{-1}}\underline{C}_{p_{c,a}} \underline{C}_{\rho_a} \kappa_* \underline{C}_{\lambda_a}.
				%		\end{equation}
			%
			\label{lem:coercivity_aqueous_saturation}
		\end{lemma}
		
		\begin{proof}
			Using \cref{eq:ah_aqueous_saturation}, \cref{eq:bound_consistency_aqueous_saturation}, the trace inequality \cref{eq:trace_ineq} and the fact that for all $K \in \mathcal{E}_h$ and all $e \in \partial K$, $h_e \leq h_K$, and that $\theta_{s_a} - 1 \leq 0$ we have
			\begin{equation}
				\begin{split}
					b_{a}^{n}(w_h,w_h) &= \sum_{K \in \mathcal{E}_h}\int_K \kappa\lambda_a^n \del[1]{\partial_{s_a}p_{c,a}}^{+,n} \abs[1]{\nabla w_h}^2
					+ \sum_{e \in \Gamma_h }\alpha_{s_a,e}h_e^{-1}\int_e \eta_{s_a,e}^n \del[1]{\sbr[1]{w_h}}^2
					\\
					&\quad + (\theta_{s_a} - 1) \sum_{e \in \Gamma_h}\int_e \cbr{\kappa \lambda_a^n \del[1]{\partial_{s_a}p_{c,a}}^{+,n} \nabla w_h \cdot \boldsymbol{n}_e}\sbr[1]{w_h} 
					\\
					&\geq \underline{C}_{p_{c,a}} \kappa_* \underline{C}_{\lambda_a} \norm[1]{\nabla w_h}^2 + \alpha_{s_a,*} \frac{\underline{C}_{p_{c,a}}^2 \underline{C}_{\lambda_a}^2 \kappa_*^2}{\overline{C}_{p_{c,a}} \overline{C}_{\lambda_a} \kappa^*} \abs[1]{w_h}_{\text{J}}^2
					+ (\theta_{s_a} - 1) \overline{C}_{\lambda_a} \kappa^* \overline{C}_{p_{c,a}} C_{\text{tr}} \norm[1]{\nabla w_h} \abs[1]{w_h}_{\text{J}}.
				\end{split}
				\label{eq:ah_aqueous_saturation1}
			\end{equation}
			Using \cref{eq:important_ineq} with 
			$c = \alpha_{s_a,*} \underline{C}_{p_{c,a}} \underline{C}_{\lambda_a} \kappa_*\del[1]{\overline{C}_{p_{c,a}} \overline{C}_{\lambda_a} \kappa^*}^{-1}$, 
			$\beta = 0.5(1 - \theta_{s_a}) \overline{C}_{\lambda_a} \kappa^* \overline{C}_{p_{c,a}} C_{\text{tr}}$ $ \del[1]{\underline{C}_{p_{c,a}} \kappa_* \underline{C}_{\lambda_a}}^{-1}$, 
			$x = \del[1]{\underline{C}_{p_{c,a}} \kappa_* \underline{C}_{\lambda_a}}^{1/2} \norm[1]{\nabla_h w_h}$, 
			and $y = \del[1]{\underline{C}_{p_{c,a}} \kappa_* \underline{C}_{\lambda_a}}^{1/2} \abs[1]{w_h}_{\text{J}}$, concludes the proof.
		\end{proof}
		
		Now, we prove that $b_{a}$ is bounded.
		
		\begin{lemma}[Boundedness of $b_{a}$]
			There exists a constant $C_{\text{B},s_a} > 0$ independent of $h$ such that, for all $n\geq 0$, $v_h \in X_{h,\Gamma_\D^{s_a}}$ and $w_h \in X_{h,\Gamma_\D^{s_a}}$, the following relation is satisfied
			\begin{equation}
				\abs[1]{b_{a}^n(v_h, w_h)} \leq C_{\text{B},s_a}\vertiii{v_h}\,\, \vertiii{w_h}.
				\label{eq:boundedness_aqueous_saturation}
			\end{equation}
			In addition, there exists a constant $C_{\text{B}_*,s_a} > 0$ independent of $h$ and $\tau$ such that for any  $v \in H^2(\Omega)+X_{h,\Gamma_\D^{s_a}}$ and any $w_h\in X_{h,\Gamma_\D^{s_a}}$, the following bound holds
			\begin{equation}
				\abs[1]{b_{a}^n(v,w_h)} \leq C_{\text{B}_*,s_a} \vertiii{v}_{*}\vertiii{w_h}.
				\label{eq:boundedness_continuous_saturation}
			\end{equation}
			\label{lem:boundedness_aqueous_saturation}
		\end{lemma}
		
		\begin{proof}
			The proof of this lemma is analogous to the one of \cref{lem:boundedness_pressure} and therefore is omitted. 
		\end{proof}
		
		\begin{corollary}
			There exists a unique solution to problem \cref{eq:disc_Sa}.
		\end{corollary}
		
		\begin{proof}
			The coercivity \cref{lem:coercivity_aqueous_saturation} and boundedness \cref{lem:boundedness_aqueous_saturation} of $a_{s_a}$, together with the fact that $\phi$ and $\rho_a$ are strictly positive, imply, using the Lax--Milgram theorem, that problem \cref{eq:disc_Sa} is well-posed.
		\end{proof}
		
		%---------------------------------------------
		\subsection{Vapor saturation}
		%---------------------------------------------
		
		%	We define the function $f_v^\ast(w_h; p_\ell^{n+1}, s_a^{n+1}, s_v^{n+1})$ as follows:
		%	\begin{equation}
			%		\begin{split}
				%			f_v^\ast(w_h;p_\ell^{n+1}, s_{a}^{n+1}, s_{v}^{n+1}) = (\tilde{\rho}_v^{n+1} q_v^{n+1}, w_h) +
				%			\sum_{K \in \mathcal{E}_h}\int_K \del{\tilde{\rho}_v^{n+1} \tilde{\lambda}_v^{n+1} (-\kappa \nabla p_\ell^{n+1}) 
					%				+ (\tilde{\rho}_v^{n+1})^2 \tilde{\lambda}_v^{n+1} \boldsymbol{g}}\cdot \nabla w_h 
				%			\\
				%			+ \sum_{e \in \Gamma_{\text{N}}^{s_v}}\int_e j_{s_v}^{\text{N}} w_h 
				%			- \sum_{e \in \Gamma_h }\int_e \tilde{\rho}_v^{n+1} \tilde{\lambda}_v^{n+1} (-\kappa \nabla p_\ell^{n+1}) \cdot\boldsymbol{n}_e \sbr[1]{w_h} 
				%			- \sum_{e \in \Gamma_h }\int_e \cbr{ (\tilde{\rho}_v^{n+1})^2\kappa \tilde{\lambda}_v^{n+1}\boldsymbol{g}\cdot\boldsymbol{n}_e} \sbr[1]{w_h}.
				%		\end{split}
			%	\end{equation}
		
		\begin{lemma}[Consistency of $b_{v}$]
			\begin{equation}
				(\phi\del[0]{\partial_t s_v}^{n+1},w_h)
				+ \tilde{b}^{n+1}_{v}(s_v^{n+1}, w_h) = \tilde{f}^{n+1}_{v}(w_h), \quad \forall w_h \in X_h, \quad \forall n\geq 0,
				\label{eq:consistency_vapor_saturation}
			\end{equation}
			where $\tilde{b}^{n+1}_{v}(s_v,w_h) = b_{v}(s_v,w_h;p_\ell^{n+1}, s_a^{n+1}, s_v^{n+1})$, and $\tilde{f}^{n+1}_{v}(w_h) = f_{v}(w_h; p_\ell^{n+1}, s_a^{n+1}, s_v^{n+1})$.
		\end{lemma}
		
		\begin{proof}
			The proof is similar to the other consistency lemma and therefore not shown here.
		\end{proof}
		
		\begin{lemma}
			For all $\del[0]{v_h, w_h} \in X_h \times X_h$ and any $n\geq 0$, the following relation is satisfied
			\begin{multline}
				\abs{\sum_{e \in \Gamma_h } \int_e \cbr{\kappa \lambda_v^n \del[1]{\partial_{s_v}p_{c,v}}^{n} \nabla v_h \cdot \boldsymbol{n}_e}\sbr[1]{w_h}} 
				\\
				\leq \overline{C}_{\lambda_v} \kappa^* \overline{C}_{p_{c,v}} \del{\sum_{K \in \mathcal{E}_h} \sum_{e \in \partial K} h_e \norm[1]{\nabla v_h|_K \cdot \boldsymbol{n}_e}_{L^2(e)}^2 }^{1/2}\abs[1]{w_h}_{\text{J}}.
				\label{eq:bound_consistency_vapor_saturation}
			\end{multline}
		\end{lemma}
		
		\begin{proof}
			The proof of this lemma is completely analogous to the proof of \cref{eq:bound_consistency_pressure}.
		\end{proof}
		
		Now we can show that $b_{v}$ is coercive on $X_h$.
		
		\begin{lemma}[Coercivity of $b_{v}$]
			Assume that $\alpha_{s_v,*}$ satisfies that
			\begin{equation}
				\alpha_{s_v,*} > 0.25\del{1 - \theta_{s_v}}^2 \del[1]{\overline{C}_{\lambda_v} \overline{C}_{p_{c,v}} \kappa^*}^3 \del[1]{\underline{C}_{\lambda_v} \underline{C}_{p_{c,v}} \kappa_*}^{-3} C_{\text{tr}}^2,
			\end{equation}
			where $C_{\text{tr}}$ results from the trace inequality \cref{eq:trace_ineq}. Then, the bilinear form $b_{v}^n$ defined by \cref{eq:ah_vapor_saturation} is coercive on $X_{h,\Gamma_\D^{s_v}}$ with respect to the norm $\vertiii[1]{\cdot}$ defined by \cref{eq:dif_norm}, i.e., for all $w_h \in X_{h,\Gamma_\D^{s_v}}$, the following relation is satisfied:
			\begin{equation}
				b_{v}^n(w_h,w_h) \geq C_{\alpha,s_v}\vertiii[1]{w_h},
				\label{eq:coercivity_vapor_saturation}
			\end{equation}
			with
			\begin{multline}
				C_{\alpha,s_v} = \frac{\alpha_{s_v,*}\underline{C}_{p_{c,v}}\underline{C}_{\lambda_v} \kappa_*\del[1]{\overline{C}_{p_{c,v}} \overline{C}_{\lambda_v} \kappa^*}^{-1} - 0.25\del{1 - \theta_{s_v}}^2 		\del[1]{\overline{C}_{p_{c,v}} \overline{C}_{\lambda_v} \kappa^*}^2\del[1] {\underline{C}_{p_{c,v}}\underline{C}_{\lambda_v} \kappa_*}^{-2} C_{\text{tr}}^2}{1 + \alpha_{s_v,*}\underline{C}_{p_{c,v}} \underline{C}_{\lambda_v} \kappa_*\del[1]{\overline{C}_{p_{c,v}} \overline{C}_{\lambda_v} \kappa^*}^{-1}}
				%
				%			C_{\alpha,s_v} = \del[2]{1 + \alpha_{s_v,*}\underline{C}_{p_{c,v}} \underline{C}_{\lambda_v} \underline{C}_{\rho_v} \kappa_*\del[1]{\overline{C}_{p_{c,v}} \overline{C}_{\lambda_v} \overline{C}_{\rho_v} \kappa^*}^{-1}}^{-1} \del[2]{\underline{C}_{p_{c,v}} \underline{C}_{\rho_v} \kappa_* \underline{C}_{\lambda_v}}
				%			\\
				%			\del[2]{0.25\del{1 - \theta_{s_v}}^2 \del[1]{\overline{C}_{\rho_v}\overline{C}_{\lambda_v} \overline{C}_{p_{c,v}} \kappa^*}^2 \del[1]{\underline{C}_{\rho_v} \underline{C}_{\lambda_v} \underline{C}_{p_{c,v}} \kappa_*}^{-2} C_{\text{tr}}^2}.
			\end{multline}
			%		\begin{equation}
				%			C_{\alpha,s_v} = \frac{\alpha_{s_v,*}\underline{C}_{p_{c,v}} \underline{C}_{\lambda_v} \underline{C}_{\rho_v} \kappa_*\del[1]{\overline{C}_{p_{c,v}} \overline{C}_{\lambda_v} \overline{C}_{\rho_v} \kappa^*}^{-1} - 0.25\del{1 - \theta_{s_v}}^2 \del[1]{\overline{C}_{\rho_v}\overline{C}_{\lambda_v} \overline{C}_{p_{c,v}} \kappa^*}^2 \del[1]{\underline{C}_{\rho_v} \underline{C}_{\lambda_v} \underline{C}_{p_{c,v}} \kappa_*}^{-2} C_{\text{tr}}^2}{1 + \alpha_{s_v,*}\underline{C}_{p_{c,v}} \underline{C}_{\lambda_v} \underline{C}_{\rho_v} \kappa_*\del[1]{\overline{C}_{p_{c,v}} \overline{C}_{\lambda_v} \overline{C}_{\rho_v} \kappa^*}^{-1}}\underline{C}_{p_{c,v}} \underline{C}_{\rho_v} \kappa_* \underline{C}_{\lambda_v}.
				%		\end{equation}
			%
			\label{lem:coercivity_vapor_saturation}
		\end{lemma}
		
		\begin{proof}
			The proof is analogous to that of \cref{eq:coercivity_aqueous_saturation} and therefore not shown here.
		\end{proof}
		
		Now, we prove that $b_{v}$ is bounded.
		
		\begin{lemma}[Boundedness of $b_{v}$]
			There exists a constant $C_{\text{B},s_v} > 0$ independent of $h$ such that, for all $n\geq 0$, $v_h \in X_{h,\Gamma_\D^{s_v}}$ and $w_h \in X_{h,\Gamma_\D^{s_v}}$, the following relation is satisfied
			\begin{equation}
				\abs[1]{b_{v}^n(v_h, w_h)} \leq C_{\text{B},s_v}\vertiii{v_h}\,\, \vertiii{w_h}.
				\label{eq:boundedness_vapor_saturation}
			\end{equation}
			In addition, there exists a constant $C_{\text{B}_*,s_v} > 0$ independent of $h$ and $\tau$ such that for any  $v \in H^2(\Omega)+X_{h,\Gamma_\D^{s_v}}$ and any $w_h\in X_{h,\Gamma_\D^{s_v}}$, the following bound holds
			\begin{equation}
				\abs[1]{b_{v}^n(v,w_h)} \leq C_{\text{B}_*,s_v} \vertiii{v}_{*}\vertiii{w_h}.
				\label{eq:boundedness_continuous_gas_fraction}
			\end{equation}
			\label{lem:boundedness_vapor_saturation}
		\end{lemma}
		
		\begin{proof}
			The proof of this lemma follows the one of \cref{lem:boundedness_pressure} and therefore is not shown here.
		\end{proof}
		
		\begin{corollary}
			There exists a unique solution to the discrete problem \cref{eq:disc_Sv}.
		\end{corollary}
		
		\begin{proof}
			The coercivity \cref{lem:coercivity_vapor_saturation} and boundedness \cref{lem:boundedness_vapor_saturation} of $a_{s_v}$, together with the fact that $\phi$ and $\rho_v$ are strictly positive, imply, using the Lax--Milgram theorem, that problem \cref{eq:disc_Sv} is well-posed.
		\end{proof}
		
		%=============================================
		\section{\textit{A priori} error estimates}
		%=============================================
		\label{sec:error_analysis}
		
		In this section, we derive \textit{a priori} error estimates. 
		To do so, we introduce the following quantities:
		\begin{subequations}
			\begin{alignat}{2}
				%\begin{split}
				e_{p_h}^{n} &= P_h^{n} - \pi_{h,\Gamma_\D^{p_\ell}} p_\ell^n, \quad &e_{p_\pi}^{n} &= p_\ell^n - \pi_{h,\Gamma_\D^{p_\ell}} p_\ell^n,
				\\
				e_{a_h}^{n} &= S_{a_h}^{n} - \pi_{h,\Gamma_\D^{s_a}} s_a^n, \quad &e_{a_\pi}^{n} &= s_a^n - \pi_{h,\Gamma_\D^{s_a}} s_a^n,
				\\
				e_{v_h}^{n} &= S_{v_h}^{n} - \pi_{h,\Gamma_\D^{s_v}} s_v^n, \quad &e_{v_\pi}^{n} &= s_v^n - \pi_{h,\Gamma_\D^{s_v}} s_v^n.
				%\end{split}
			\end{alignat}
		\end{subequations}
		We can then decompose the errors as
		\begin{subequations}
			\begin{alignat}{1}
				p_\ell^n - P_h^{n} &= e_{p_\pi}^{n} - e_{p_h}^{n},
				\label{eq:error_split_pressure}
				\\
				s_a^n - S_{a_h}^{n} &= e_{p_\pi}^{n} - e_{p_h}^{n},
				\label{eq:error_split_saturation}
				\\
				s_a^n - S_{a_h}^{n} &= e_{p_\pi}^{n} - e_{p_h}^{n}.
				\label{eq:error_split_gasfraction}
			\end{alignat}
			\label{eq:error_split}
		\end{subequations}
		We note that, thanks to the definition of the $L^2$-orthogonal projection, the errors above satisfy that
		\begin{equation}
			(e_{p_\pi}^n, w_h) = (e_{a_\pi}^n, w_h) = (e_{v_\pi}^n, w_h) = 0, \quad \forall w_h \in X_h.
			\label{eq:L2_property}
		\end{equation}
		
		We will make use of the following two auxiliary lemmas.
		\begin{lemma}
			For any $0 \leq n \leq N-1$, and any $w_h \in X_h$, we have the following bounds
			\begin{subequations}
				\begin{alignat}{2}
					\abs{\tilde{b}^{n+1}_{p}(p_\ell^{n+1}, w_h) - \tilde{b}^{n}_{p}(p_\ell^{n+1}, w_h)} &\leq C \tau \vertiii{w_h},
					\label{eq:bp_n_bp_nplus1}
					\\
					\abs{\tilde{b}^{n+1}_{a}(s_a^{n+1}, w_h) - \tilde{b}^{n}_{a}(s_a^{n+1}, w_h)} &\leq C \tau \vertiii{w_h},
					\label{eq:ba_n_ba_nplus1}
					\\
					\abs{\tilde{b}^{n+1}_{v}(s_v^{n+1}, w_h) - \tilde{b}^{n}_{v}(s_v^{n+1}, w_h)} &\leq C \tau \vertiii{w_h}.
					\label{eq:bv_n_bv_nplus1}
				\end{alignat}
				\label{eq:b_bounds_time}
			\end{subequations}
			Moreover,
			\begin{subequations}
				\begin{alignat}{2}
					\abs{\tilde{f}^{n+1}_{p}(w_h) - \tilde{f}^{n}_{p}(w_h)} &\leq  C \tau \vertiii{w_h},
					\label{eq:fp_n_fp_nplus1}
					\\
					\abs{\tilde{f}^{n+1}_{a}(w_h) - \tilde{f}^{n}_{a}(w_h)} &\leq  C \tau \vertiii{w_h},
					\label{eq:fa_n_fa_nplus1}
					\\
					\abs{\tilde{f}^{n+1}_{v}(w_h) - \tilde{f}^{n}_{v}(w_h)} &\leq  C \tau \vertiii{w_h},
					\label{eq:fv_n_fv_nplus1}
				\end{alignat}
				\label{eq:f_bounds_time}
			\end{subequations}
		\end{lemma}
		
		\begin{proof}
			%	We only show \cref{eq:bp_n_bp_nplus1}. The proof for the other results is analogous. Note that
			%	%
			%	\begin{multline}
				%		\abs{\tilde{b}^{n+1}_{p}(p_\ell^{n+1}, w_h) - \tilde{b}^{n}_{p}(p_\ell^{n+1}, w_h)}
				%		\leq \sum_{K \in \mathcal{E}_h}\int_K \abs{\del[1]{\tilde{\lambda}_t^{n+1} - \tilde{\lambda}_t^{n}} \kappa \nabla p_\ell^{n+1}\cdot \nabla w_h}
				%		\\
				%		+ \sum_{e \in \Gamma_h }\int_e \abs{\cbr[1]{\left(\tilde{\lambda}_t^{n+1} - \tilde{\lambda}_t^{n}\right) \kappa\nabla p_\ell^{n+1}\cdot \boldsymbol{n}_e}\sbr[0]{w_h}}.
				%		\label{eq:716step1}
				%	\end{multline}
			%	Using the Lipschitz continuity assumption \ref{hyp_lipschitz}, and the fact that $\kappa$ is bounded above, we obtain
			%	%
			%	\begin{multline}
				%		\abs{\tilde{b}^{n+1}_{p}(p_\ell^{n+1}, w_h) - \tilde{b}^{n}_{p}(p_\ell^{n+1}, w_h)}
				%		\leq C\sum_{K \in \mathcal{E}_h}\int_K \del[1]{\abs[0]{s_a^{n+1} - s_a^n} + \abs[0]{s_v^{n+1} - s_v^n}} \abs{\nabla p_\ell^{n+1}\cdot \nabla w_h}
				%		\\
				%		+ \sum_{e \in \Gamma_h }\int_e \del[1]{\abs[0]{s_a^{n+1} - s_a^n} + \abs[0]{s_v^{n+1} - s_v^n}} \abs{\nabla p_\ell^{n+1}\cdot \boldsymbol{n}_e\sbr[0]{w_h}}.
				%		\label{eq:716step1}
				%	\end{multline}
			
			The results can be obtained using the Lipschitz continuity of all the coefficients and the smoothness of $p_\ell, \, s_a, \, s_v$.
		\end{proof}
		\begin{lemma}
			For any $0 \leq n \leq N$, and any $w_h \in X_h$, we have the following bounds
			\begin{subequations}
				\begin{alignat}{2}
					\abs{\tilde{b}^{n}_{p}(p_\ell^{n+1}, w_h) - b^{n}_{p}(p_\ell^{n+1}, w_h)} &\leq C \left(h^2 + \Vert S_{a_h}^{n} - s_a^{n} \Vert + 
					\Vert S_{v_h}^{n} - s_v^{n}\Vert\right) \vertiii{w_h},
					\label{eq:bp_bp_tilde}
					\\
					\abs{\tilde{b}^{n}_{a}(s_a^{n+1}, w_h) - b^{n}_{a}(s_a^{n+1}, w_h)} &\leq C \left(h^2 + \Vert S_{a_h}^{n} - s_a^{n} \Vert + \Vert S_{v_h}^{n} - s_v^{n} \Vert\right) \vertiii{w_h},
					\label{eq:ba_ba_tilde}
					\\
					\abs{\tilde{b}^{n}_{v}(s_v^{n+1}, w_h) - b^{n}_{v}(s_v^{n+1}, w_h)} &\leq C \left(h^2 + \Vert S_{a_h}^{n} - s_a^{n} \Vert +
					\Vert S_{v_h}^{n} - s_v^{n}\Vert\right) \vertiii{w_h}.
					\label{eq:bv_bv_tilde}
				\end{alignat}
				\label{eq:b_bounds}
			\end{subequations}
			Moreover,
			\begin{subequations}
				\begin{alignat}{2}
					\abs{\tilde{f}^{n}_{p}(w_h) - f^{n}_{p}(w_h)} &\leq  C \left(h^2 + \Vert S_{a_h}^{n} - s_a^{n} \Vert + 
					\Vert S_{v_h}^{n} - s_v^{n}\Vert\right) \vertiii{w_h},
					\label{eq:fp_fp_tilde}
					\\
					\abs{\tilde{f}^{n}_{a}(w_h) - f^{n}_{a}(w_h)} &\leq C \left(h + \vertiii{e_{p_h}^{n+1}} + \Vert S_{a_h}^{n} - s_a^{n} \Vert + \Vert S_{v_h}^{n} - s_v^{n} \Vert\right) \vertiii{w_h},
					\label{eq:fa_fa_tilde}
					\\
					\abs{\tilde{f}^{n}_{v}(w_h) - f^{n}_{v}(w_h)} &\leq C \left(h + \vertiii{e_{p_h}^{n+1}} + \Vert S_{a_h}^{n} - s_a^{n} \Vert + \Vert S_{v_h}^{n} - s_v^{n} \Vert\right) \vertiii{w_h}.
					\label{eq:fv_fv_tilde}
				\end{alignat}
				\label{eq:f_bounds}
			\end{subequations}
		\end{lemma}
		
		\begin{proof}
			We start by showing \cref{eq:bp_bp_tilde}. Note that using the definition of $b_{p}$ \cref{eq:ah_pressure}, we obtain
			\begin{multline}
				\abs{\tilde{b}^{n}_{p}(p_\ell^{n+1}, w_h) - b^{n}_{p}(p_\ell^{n+1}, w_h)}
				\leq \sum_{K \in \mathcal{E}_h}\int_K \abs{\del[1]{\tilde{\lambda}_t^{n} - \lambda_t^{n}} \kappa \nabla p_\ell^{n+1}\cdot \nabla w_h}
				\\
				+ \sum_{e \in \Gamma_h }\int_e \abs{\cbr[1]{\left(\tilde{\lambda}_t^{n} - \lambda_t^{n}\right) \kappa\nabla p_\ell^{n+1}\cdot \boldsymbol{n}_e}\sbr[0]{w_h}}.
				\label{eq:716step}
			\end{multline}
			With the Lipschitz continuity assumptions \ref{hyp_lipschitz}, we can write
			\begin{equation*}
				\abs{\tilde{\lambda}_t^{n} - \lambda_t^{n}}
				\leq C (\vert S_{v_h}^{n} - s_v^{n} \vert + \vert S_{a_h}^{n} - s_a^{n} \vert).
			\end{equation*}
			The first term in the right-hand side of \eqref{eq:716step} is bounded by %(assuming $\nabla p_\ell$ belongs to $L^\infty(\Omega\times (0,T))$:
			\begin{align*}
				\sum_{K \in \mathcal{E}_h}\int_K \abs{\left(\tilde{\lambda}_t^{n} - \lambda_t^{n}\right) \kappa \nabla p_\ell^{n+1}\cdot \nabla w_h}
				&\leq C \left(\Vert S_{a_h}^{n} - s_a^{n} \Vert + 
				\Vert S_{v_h}^{n} - s_v^{n}\Vert\right) \vertiii{w_h},
			\end{align*}
			where we have used the Cauchy--Schwarz inequality, boundedness of $\kappa$ and the smoothness of $p_\ell$. For the second term in the right-hand side of \eqref{eq:716step}, we have
			\begin{multline*}
				\sum_{e \in \Gamma_h }\int_e \abs{\cbr[1]{\left(\tilde{\lambda}_t^{n} - \lambda_t^{n}\right) \kappa\nabla p_\ell^{n+1}\cdot \boldsymbol{n}_e}\sbr[1]{w_h}} \leq C \sum_{e \in \Gamma_h }\int_e \cbr[1]{\vert S_{v_h}^{n} - s_v^{n} \vert + \vert S_{a_h}^{n} - s_a^{n} \vert}\abs{\sbr[0]{w_h}} 
				\\
				\leq C \sum_{e \in \Gamma_h }\int_e \cbr[1]{\vert e_{v_h}^n \vert + \vert e_{a_h}^n \vert}\abs{\sbr[0]{w_h}} + C \sum_{e \in \Gamma_h }\int_e \cbr[1]{\vert e_{v_\pi}^n \vert + \vert e_{a_\pi}^n \vert}\abs{\sbr[0]{w_h}}
				\\
				\leq C \sum_{e \in \Gamma_h }h_e^{1/2}\del[1]{\Vert e_{v_h}^n \Vert_{L^2(e)} + \Vert e_{a_h}^n \Vert_{L^2(e)}}h_e^{-1/2}\norm[0]{\sbr[0]{w_h}}_{L^2(e)} + C \sum_{e \in \Gamma_h } h_e^{1/2}\del[1]{\Vert e_{v_\pi}^n \Vert_{L^2(e)} + \Vert e_{a_\pi}^n \Vert_{L^2(e)}}h_e^{-1/2}\norm[0]{\sbr[0]{w_h}}_{L^2(e)}.
			\end{multline*}
			Using the trace inequality \cref{eq:trace_ineq} and the projection estimates \cref{eq:L2projerrors_face}, we have
			\begin{equation*}
				\sum_{e \in \Gamma_h }\int_e \abs{\cbr[1]{\left(\tilde{\lambda}_t^{n} - \lambda_t^{n}\right) \kappa\nabla p_\ell^{n+1}\cdot \boldsymbol{n}_e}\sbr[1]{w_h}} \leq 
				C \del[1]{\Vert e_{v_h}^n \Vert + \Vert e_{a_h}^n \Vert} \vertiii{w_h} + C h^2\vertiii{w_h}.
			\end{equation*}
			Combining these bounds, we obtain
			\[
			\abs{\tilde{b}^{n}_{p}(p_\ell^{n+1}, w_h) - b^{n}_{p}(p_\ell^{n+1}, w_h)}
			\leq C \left(h^2 + \Vert S_{a_h}^n - s_a^n  \Vert + 
			\Vert S_{v_h}^n - s_v^n \Vert\right) \vertiii{w_h}.
			\]
			The proof for \cref{eq:ba_ba_tilde} and \cref{eq:bv_bv_tilde} are analogous to that one of \cref{eq:bp_bp_tilde} and therefore not shown here.
			
			Next, we show \cref{eq:fp_fp_tilde}. Using the definition of $f_{p}$, see \cref{eq:fh_pressure}, we obtain
			\[
			\abs{f_{p}^{n}(w_h) -\tilde{f}_p^{n}(w_h)} = \abs{f_p(w_h;P_h^{n}, S_{a_h}^{n}, S_{v_h}^{n}) - f_p(w_h;p_\ell^{n}, s_{a}^{n}, s_{v}^{n})} \leq \abs{T_1} + \abs{T_2} + \abs{T_3},
			\]
			with
			\begin{align}
				T_1 =& - \sum_{K \in \mathcal{E}_h}\int_K \del[1]{\lambda_v^{n} \kappa\nabla p_{c,v}^{n} - \lambda_a^{n} \kappa\nabla p_{c,a}^{n} - \kappa \del[0]{\rho\lambda}_t^{n}\boldsymbol{g} }\cdot \nabla w_h\nonumber
				\\
				& + \sum_{K \in \mathcal{E}_h}\int_K 	\del[1]{\tilde{\lambda}_v^{n} \kappa \nabla \tilde{p}_{c,v}^{n} - \tilde{\lambda}_a^{n} \kappa\nabla \tilde{p}_{c,a}^{n} - \kappa\del[0]{\tilde{\rho}\tilde{\lambda}}_t^{n}\boldsymbol{g} }\cdot \nabla w_h,
				\\
				T_2 =& \sum_{e \in \Gamma_h }\int_e \cbr[1]{\lambda_v^{n} \kappa \nabla p_{c,v}^{n} \cdot \boldsymbol{n}_e - \tilde{\lambda}_v^{n} \kappa \nabla \tilde{p}_{c,v}^{n} \cdot\boldsymbol{n}_e}\sbr[1]{w_h} \nonumber
				\\
				& - \sum_{e \in \Gamma_h }\int_e \cbr[1]{\lambda_a^{n} \kappa\nabla p_{c,a}^{n} \cdot \boldsymbol{n}_e - \tilde{\lambda}_a^{n} \kappa\nabla \tilde{p}_{c,a}^{n} \cdot \boldsymbol{n}_e}\sbr[1]{w_h},
				\\
				T_3 =& - \sum_{e \in \Gamma_h }\int_e \cbr[1]{\kappa\lambda_t^{n} \boldsymbol{g} \cdot \boldsymbol{n}_e}\sbr[1]{w_h}\nonumber + \sum_{e \in \Gamma_h }\int_e 	\cbr[1]{\kappa\tilde{\lambda}_t^{n} \boldsymbol{g} \cdot \boldsymbol{n}_e}\sbr[1]{w_h}.
			\end{align}
			Using boundedness and growth conditions of $\nabla p_{c,a}$, $\nabla p_{c,v}$ (see \ref{hyp_pcgradients}), boundedness of the rest of the coefficients, and \ref{hyp_lipschitz}, the volume terms $T_1$ are bounded as
			\[
			\abs{T_1} \leq  C \left(\Vert S_{a_h}^{n} - s_a^{n} \Vert + 
			\Vert S_{v_h}^{n} - s_v^{n}\Vert\right) \vertiii{w_h}.
			\]
			With the same assumptions, for the face terms $T_2$ we have
			\begin{multline}
				\abs{T_2} \leq C\sum_{e \in \Gamma_h }\int_e 	\del[1]{\vert S_{a_h}^n - s_a^n\vert + \vert S_{v_h}^n - s_v^n\vert}\abs[0]{\sbr[0]{w_h}} = C\sum_{e \in \Gamma_h }\int_e h_e^{1/2}\del[1]{\vert e_{a_h}^n\vert + \vert e_{v_h}^n\vert}h_e^{-1/2}\abs[0]{\sbr[0]{w_h}} \\
				+ C\sum_{e \in \Gamma_h }\int_e h_e^{1/2}\del[1]{\vert e_{a_\pi}^n\vert + \vert e_{v_\pi}^n\vert}h_e^{-1/2}\abs[0]{\sbr[0]{w_h}}
				\leq C \left(\Vert S_{a_h}^{n} - s_a^{n} \Vert + 
				\Vert S_{v_h}^{n} - s_v^{n}\Vert\right) \vertiii{w_h} + C h^2\vertiii{w_h},
			\end{multline}
			where we have used the trace inequality \cref{eq:trace_ineq}, \cref{eq:error_split} and the projection estimates \cref{eq:L2projerrors_face}.
			Similarly, the terms in $T_3$ are bounded by
			\[
			\abs{T_3} \leq C \left(h^2 + \Vert S_{a_h}^{n} - s_a^{n} \Vert + 
			\Vert S_{v_h}^{n} - s_v^{n}\Vert\right) \vertiii{w_h},
			\]
			which concludes the proof.
			
			To prove \cref{eq:fa_fa_tilde}, note that 
			\[
			\abs{f_a^{n}(w_h) - \tilde{f}_a^{n}(w_h)} = \abs{f_a(w_h;P_h^{n+1},S_{a_h}^{n},S_{v_h}^{n})-f_a(w_h;p_\ell^{n+1},s_a^{n},s_v^{n})}
			\leq \abs{T_1}+\abs{T_2}+\abs{T_3},
			\]
			with
			\begin{align*}
				T_1 =&  \sum_{K \in \mathcal{E}_h}\int_K \del{\lambda_a^{n} \boldsymbol{u}_h^{n+1} + \rho_a \lambda_a^{n} \boldsymbol{g}}\cdot \nabla w_h 
				- \sum_{K \in \mathcal{E}_h}\int_K \del{\tilde{\lambda}_a^{n}\boldsymbol{u}^{n+1} + \rho_a \tilde{\lambda}_a^{n} \boldsymbol{g}}\cdot \nabla w_h,
				\\
				T_2 =& - \sum_{e \in \Gamma_h }\int_e \left( \lambda_a^{n}\right)_{s_a}^\uparrow  \boldsymbol{u}_h^{n+1} \cdot\boldsymbol{n}_e  \sbr[1]{w_h} 
				+ \sum_{e \in \Gamma_h }\int_e  \tilde{\lambda}_a^{n} \boldsymbol{u}^{n+1} \cdot\boldsymbol{n}_e  \sbr[1]{w_h},
				\\
				T_3 =& - \sum_{e \in \Gamma_h }\int_e \cbr{ \rho_a\kappa \lambda_a^{n} \boldsymbol{g}\cdot\boldsymbol{n}_e} \sbr[1]{w_h}
				+ \sum_{e \in \Gamma_h }\int_e \cbr{ \rho_a\kappa \tilde{\lambda}_a^{n} \boldsymbol{g}\cdot\boldsymbol{n}_e} \sbr[1]{w_h},
			\end{align*}
			where we recall that $\boldsymbol{u}_h^{n+1} = \Pi_\mathrm{RT}(-\kappa \nabla P_h^{n+1})$ and
			we denote: $\boldsymbol{u}^{n+1} = -\kappa \nabla p_\ell^{n+1}$.
			Note that we have
			\[
			\abs{T_1} \leq \sum_{K \in \mathcal{E}_h}\int_K \abs{\del{ \lambda_a^{n} \boldsymbol{u}_h^{n+1} - \tilde{\lambda}_a^{n} \boldsymbol{u}^{n+1}}\cdot \nabla w_h }
			+ \sum_{K \in \mathcal{E}_h}\int_K \abs{\del{\rho_a \lambda_a^{n} \boldsymbol{g} - \rho_a \tilde{\lambda}_a^{n} \boldsymbol{g}}\cdot \nabla w_h}.
			\]
			The second term is bounded by:
			\[
			\sum_{K \in \mathcal{E}_h}\int_K \abs{\del{\rho_a \lambda_a^n \boldsymbol{g} - \rho_a \tilde{\lambda}_a^{n} \boldsymbol{g}}\cdot \nabla w_h}
			\leq C\left(\Vert S_{a_h}^n - s_a^n \Vert + \Vert S_{v_h}^n - s_v^n \Vert\right) \vertiii{w_h}.
			\]
			For the first term in $T_1$, we have
			\begin{equation}
				\sum_{K \in \mathcal{E}_h}\int_K \lambda_a^{n} \abs{\del{\boldsymbol{u}_h^{n+1} - \boldsymbol{u}^{n+1}}\cdot \nabla w_h} + \sum_{K \in \mathcal{E}_h}\int_K \abs{( \lambda_a^{n+1} - \tilde{\lambda}_a^{n+1})\boldsymbol{u}^{n+1} \cdot \nabla w_h}.
				\label{eq:aux_T3}
			\end{equation}
			We write
			\[
			\sum_{K \in \mathcal{E}_h}\int_K \abs{\lambda_a^{n} \del{\boldsymbol{u}_h^{n+1} - \boldsymbol{u}^{n+1}}\cdot \nabla w_h}  \leq C \Vert \boldsymbol{u}_h^{n+1} - \boldsymbol{u}^{n+1} \Vert \, \vertiii{w_h}
			\]
			Using triangle inequality, \cref{eq:L2projerrors} and \eqref{eq:RT_estimate}, we have
			\[
			\Vert \boldsymbol{u}_h^{n+1} - \boldsymbol{u}^{n+1} \Vert
			\leq \Vert \boldsymbol{u}_h^{n+1} +\kappa  \nabla_h P_h^{n+1} \Vert
			+ \Vert \kappa \nabla_h (P_h^{n+1} - p_\ell^{n+1}) \Vert
			\leq C \vertiii{e_{p_h}^{n+1}} + C h.
			\]
			For the second term in \cref{eq:aux_T3}, we have
			\[
			\sum_{K \in \mathcal{E}_h}\int_K \abs{(\lambda_a^{n} - \tilde{\lambda}_a^{n}) \kappa \nabla p_\ell^{n+1}  \cdot \nabla w_h}
			\leq C\left(\Vert S_{a_h}^n - s_a^n \Vert + \Vert S_{v_h}^n - s_v^n \Vert\right) \vertiii{w_h},
			\]
			where we have used that $p_\ell \in C^0(0,T; H^{3}(\Omega))$.
			So combining the bounds above, we obtain
			\[
			\abs{T_1} \leq C \left(h + \vertiii{e_{p_h}^{n+1}} + \Vert S_{a_h}^n - s_a^n \Vert + \Vert S_{v_h}^n - s_v^n \Vert\right)\vertiii{w_h}.
			\]
			The term $T_2$ can be written as
			\begin{equation}
				\abs{T_2} \leq \sum_{e \in \Gamma_h }\int_e \abs{\left( \lambda_a^{n}\right)_{s_a}^\uparrow  (\boldsymbol{u}_h^{n+1} - \bfu^{n+1}) \cdot\boldsymbol{n}_e  \sbr[0]{w_h}}
				+ \sum_{e \in \Gamma_h }\int_e \abs{\left(\left( \lambda_a^{n}\right)_{s_a}^\uparrow 
					- \tilde{\lambda}_a^{n} \right)\bfu^{n+1} \cdot\boldsymbol{n}_e  \sbr[0]{w_h}}.
				\label{eq:aux_T5}
			\end{equation}
			The first term in \cref{eq:aux_T5} is bounded by
			\[
			\sum_{e \in \Gamma_h }\int_e \abs{\left( \lambda_a^{n+1} \right)_{s_a}^\uparrow  (\boldsymbol{u}_h^{n+1} - \bfu^{n+1}) \cdot\boldsymbol{n}_e  \sbr[0]{w_h} }
			\leq C \vertiii{w_h} \left(\sum_{e\in\Gamma_h} h_e \Vert \bfu_h^{n+1} - \bfu^{n+1}\Vert_{L^2(e)}^2\right)^{1/2}.
			\]
			We fix a face $e$ and we choose one neighboring element $K_e$ such that $e\subset \partial K_e$. 
			\[
			\Vert \bfu_h^{n+1} -\bfu^{n+1} \Vert_{L^2(e)} \leq \Vert \bfu_h^{n+1} + \kappa \nabla P_h^{n+1}|_{K_e}\Vert_{L^2(e)}
			+ \Vert \kappa \nabla (P_h^{n+1}|_{K_e} - p_\ell^{n+1})\Vert_{L^2(e)}.
			\]
			Using the trace inequality \cref{eq:trace_ineq} and the Raviart--Thomas projection estimate \eqref{eq:RT_estimate}, we have
			\[
			\left(\sum_{e\in\Gamma_h} h_e \Vert \bfu_h^{n+1} + \kappa \nabla P_h^{n+1}|_{K_e}\Vert_{L^2(e)}^2 \right)^{1/2}
			\leq \Vert \bfu_h^{n+1} + \kappa \nabla_h P_h^{n+1}\Vert \leq C \vertiii{e_{p_h}^{n+1}} + Ch.
			\]
			Adding and subtracting $\pi_{h,\Gamma_\D^{p_\ell}} p_\ell^{n+1}$ and using \cref{eq:L2projerrors_gradface} and the trace inequality \cref{eq:trace_ineq}, yields
			\[
			\left(\sum_{e\in\Gamma_h} h_e \Vert \nabla (P_h^{n+1}|_{K_e}-p_\ell^{n+1})\Vert_{L^2(e)}^2\right)^{1/2}
			\leq C \vertiii{e_{p_h}^{n+1}} + C h.
			\]
			Therefore, the first term in \cref{eq:aux_T5} is bounded as
			\begin{eqnarray*}
				\sum_{e \in \Gamma_h }\int_e \abs{\left( \lambda_a^{n} \right)_{s_a}^\uparrow  (\boldsymbol{u}_h^{n+1} - \bfu^{n+1}) \cdot\boldsymbol{n}_e  \sbr[0]{w_h} }
				\leq & 
				C (h+\vertiii{e_{p_h}^{n+1}})\vertiii{w_h}.
			\end{eqnarray*}
			The second term in \cref{eq:aux_T5} can be bounded as:
			\begin{equation*}
				\sum_{e \in \Gamma_h }\int_e \abs{\left(\left(\lambda_a^{n}\right)_{s_a}^\uparrow 
					- \tilde{\lambda}_a^{n} \right)\bfu^{n+1} \cdot\boldsymbol{n}_e  \sbr[0]{w_h}}
				\leq C \left(h^2 + \Vert S_{a_h}^n - s_a^n \Vert + \Vert S_{v_h}^n - s_v^n \Vert\right) \, \vertiii{w_h}.
			\end{equation*}
			Therefore combining the bounds above, and using that $h\leq 1$ so that $h^2 \leq h$, we have
			\[
			\abs{T_2} \leq 
			C \left(h + \vertiii{e_{p_h}^{n+1}}
			+  \Vert S_{a_h}^n - s_a^n \Vert + \Vert S_{v_h}^n - s_v^n \Vert\right)\vertiii{w_h}.
			\]
			Finally, the term $T_3$ is bounded by
			\[
			\abs{T_3}
			\leq C\left(h^2 + \Vert S_{a_h}^n - s_a^n \Vert + \Vert S_{v_h}^n - s_v^n \Vert\right) \, \vertiii{w_h}.
			\]
			Combining all the bounds above gives the result.
			
			The proof for \cref{eq:fv_fv_tilde} is analogous to that of \cref{eq:fa_fa_tilde}.
		\end{proof}
		
		%---------------------------------------------
		\subsection{Liquid pressure}
		%---------------------------------------------

		The following lemma gives an equation for the error $e_{p_h}^{n}$.
		
		\begin{lemma}[Error equation for the liquid pressure]
			We have that, for all $w_h \in X_{h,\Gamma_\D^{p_\ell}}$, and all $0 \leq n \leq N-1$, there exists a constant $C>0$ independent of $h$ and $\tau$ such that
			\begin{equation}
				b_{p}^{n}(e_{p_h}^{n+1},w_h)
				\leq
				b_{p}^{n}(e_{p_\pi}^{n+1},w_h) 
				+ C \left(\tau + h^2 + \Vert S_{a_h}^{n} - s_a^n \Vert + \Vert S_{v_h}^{n} - s_v^n \Vert \right) \vertiii{w_h}.
				\label{eq:error_eq_pressure_approach1}
			\end{equation}
			\label{lem:error_eq_pressure}
		\end{lemma}
		
		\begin{proof}
			Subtracting the consistency of the scheme \cref{eq:consistency_pressure} from the discretization \cref{eq:disc_pl}, we have
			\begin{equation}
				b_{p}^{n}(P_{h}^{n+1},w_h) - \tilde{b}^{n+1}_{p}(p_\ell^{n+1}, w_h) = f_{p}^{n}(w_h) - \tilde{f}^{n+1}_{p}(w_h).
			\end{equation}
			This is equivalent to
			\begin{equation}
				b_{p}^{n}(P_{h}^{n+1},w_h) - \tilde{b}^{n}_{p}(p_\ell^{n+1}, w_h) = f_{p}^{n}(w_h) - \tilde{f}^{n}_{p}(w_h) + \tilde{f}^{n}_{p}(w_h) - \tilde{f}^{n+1}_{p}(w_h) - \tilde{b}^{n}_{p}(p_\ell^{n+1}, w_h) + \tilde{b}^{n+1}_{p}(p_\ell^{n+1}, w_h).
			\end{equation}
			Thanks to \cref{eq:bp_n_bp_nplus1} and \cref{eq:fp_n_fp_nplus1}, we have
			\begin{equation}
				b_{p}^{n}(P_{h}^{n+1},w_h) - \tilde{b}^{n}_{p}(p_\ell^{n+1}, w_h) = f_{p}^{n}(w_h) - \tilde{f}^{n}_{p}(w_h) + C\tau \vertiii{w_h}.
			\end{equation}
			This is also equivalent to
			\begin{equation}
				b_{p}^{n}(P_{h}^{n+1},w_h) - b^{n}_{p}(p_\ell^{n+1}, w_h) = f_{p}^{n}(w_h) - \tilde{f}^{n}_{p}(w_h) + C\tau \vertiii{w_h} + \tilde{b}^{n}_{p}(p_\ell^{n+1}, w_h) - b^{n}_{p}(p_\ell^{n+1}, w_h).
			\end{equation}
			Owing to \cref{eq:bp_bp_tilde}, \cref{eq:fp_fp_tilde}, this is equivalent to
			\begin{equation}
				b_{p}^{n}(P_{h}^{n+1} - p_\ell^{n+1},w_h) \leq C \left(\tau + h^2 + \Vert S_{a_h}^{n} - s_a^n \Vert + \Vert S_{v_h}^{n} - s_v^n \Vert \right) \vertiii{w_h}.
				\label{eq:estimates3_1}
			\end{equation}
			The result is obtained by using \cref{eq:error_split_pressure}.
		\end{proof}
		
		\begin{lemma}[Error estimates for the liquid pressure]
			We have that, for all $0 \leq n \leq N-1$,
			\begin{equation}			
				\tilde{C} \tau \vertiii{e_{p_h}^{n+1}}^2
				\leq C\tau(\tau^2 + h^2)
				+ C \tau \left(\Vert S_{a_h}^{n} - s_a^n \Vert^2 + \Vert S_{v_h}^{n} - s_v^n \Vert^2 \right),
				\label{eq:error_estimate_pressure}
			\end{equation}
			where $C, \tilde{C} > 0$ are independent of $h$ and $\tau$.
			\label{lem:error_estimate_pressure}
		\end{lemma}
		
		\begin{proof}
			Letting $w_h = \tau e_{p_h}^{n+1}$ in \cref{eq:error_eq_pressure_approach1}, it reads
			\begin{equation}
				\tau b_{p}^{n}(e_{p_h}^{n+1},e_{p_h}^{n+1})
				\leq
				\tau b_{p}^{n}(e_{p_\pi}^{n+1},e_{p_h}^{n+1}) 
				+ C \tau \left(\tau + h^2 + \Vert S_{a_h}^{n} - s_a^n \Vert + \Vert S_{v_h}^{n} - s_v^n \Vert \right) \vertiii{e_{p_h}^{n+1}}.
			\end{equation}
			Applying coercivity \cref{eq:coercivity_pressure} and boundedness \cref{eq:boundedness_continuous_pressure} of $b_{p}^n$, we obtain
			\begin{equation}
				C_{\alpha,p_\ell} \tau \vertiii{e_{p_h}^{n+1}}^2
				\leq
				C_{\mathrm{B}_\ast,p_\ell} \tau \vertiii{e_{p_\pi}^{n+1}}_\ast \vertiii{e_{p_h}^{n+1}}
				+ C \tau \left(\tau + h^2 + \Vert S_{a_h}^{n} - s_a^n \Vert + \Vert S_{v_h}^{n} - s_v^n \Vert \right) \vertiii{e_{p_h}^{n+1}}.
				\label{eq:pl_est_1}
			\end{equation}
			Using Young's inequality on the right hand side, we have
			\begin{equation}
				C \tau \vertiii{e_{p_h}^{n+1}}^2
				\leq
				C \tau \vertiii{e_{p_\pi}^{n+1}}_\ast^2 
				+ C \tau \left(\tau^2 + h^4 + \Vert S_{a_h}^{n} - s_a^n \Vert^2 + \Vert S_{v_h}^{n} - s_v^n \Vert^2 \right).
				\label{eq:pl_est_2}
			\end{equation}
			The result is obtained after using the $L^2$-orthogonal projection estimates \cref{eq:StabNorm_L2projerrors} and recalling that $h^4 \leq h^2$.
		\end{proof}
		
		%---------------------------------------------
		\subsection{Aqueous saturation}
		%---------------------------------------------
		The following lemma gives an equation for the error $e_{a_h}^{n}$.
		
		\begin{lemma}[Error equation for the aqueous saturation]
			We have that, for all $w_h \in X_{h,\Gamma_\D^{s_a}}$, and all $0 \leq n \leq N-1$, there exists a constant $C>0$ independent of $h$ and $\tau$ such that
			\begin{multline}
				\dfrac{1}{\tau} (\phi (S_{a_h}^{n+1} - s_{a}^{n+1}), w_h) 
				+ b_{a}^{n}(e_{a_h}^{n+1}, w_h) 
				= \dfrac{1}{\tau} (\phi (S_{a_h}^n - s_{a}^n), w_h) + b_{a}^{n}(e_{a_\pi}^{n+1}, w_h)
				\\
				+ C \left(\tau + h + \vertiii{e_{p_h}^{n+1}} + \Vert S_{a_h}^{n} - s_a^{n} \Vert + \Vert S_{v_h}^{n} - s_v^{n} \Vert\right) \vertiii{w_h} + \sigma_a(w_h),
				\label{eq:error_eq_aqueous_saturation}
			\end{multline}
			where
			\begin{equation}
				\sigma_a(w_h) = \frac{1}{\tau} (\beta_a^{n+1},w_h), \quad \beta_a^{n+1} = \phi \int_{t_n}^{t_{n+1}}(t-t_n)\partial_{tt}s_a \dif t.
				\label{eq:beta_def}
			\end{equation}
			\label{lem:error_eq_aqueous_saturation}
		\end{lemma}
		
		\begin{proof}
			Using a Taylor series expansion, we have
			\begin{equation}
				s_a^n = s_a^{n+1} - \tau \del[0]{\partial_t s_a}^{n+1} + \int_{t_n}^{t_{n+1}}\del[0]{t-t_n}\partial_{tt} s_a \dif t,
			\end{equation}
			Rearranging the terms, multiplying by $\phi$ and a test function $w_h \in X_h$, and integrating over $\Omega$ yields
			\begin{equation}
				\dfrac{1}{\tau}(\phi s_a^{n+1}, w_h) = \dfrac{1}{\tau} (\phi s_a^n, w_h) + (\phi(\partial_t s_a)^{n+1}, w_h) - \sigma_a(w_h),
			\end{equation}
			where $\sigma_a(w_h)$ is defined in \cref{eq:beta_def}.  Using the consistency of the scheme \cref{eq:consistency_saturation} on the second term on the right-hand side, rearranging terms, and subtracting the result from \cref{eq:disc_Sa} reads
			\begin{multline}
				\dfrac{1}{\tau} (\phi (S_{a_h}^{n+1} - s_{a}^{n+1}), w_h) 
				+ b_{a}^{n}(S_{a_h}^{n+1}, w_h)
				- \tilde{b}_{a}^{n+1}(s_a^{n+1}, w_h) 
				= \dfrac{1}{\tau} (\phi (S_{a_h}^n - s_{a}^n), w_h) 
				\\
				+ f_{a}^{n}(w_h) - \tilde{f}_{a}^{n+1}(w_h) + \sigma_a(w_h).
				\label{eq:sa_aux1}
			\end{multline}
			Owing to \cref{eq:ba_ba_tilde} and \cref{eq:fa_fa_tilde}, this is equivalent to
			\begin{multline}
				\dfrac{1}{\tau} (\phi (S_{a_h}^{n+1} - s_{a}^{n+1}), w_h) 
				+ b_{a}^{n}(S_{a_h}^{n+1} - s_a^{n+1}, w_h) 
				= \dfrac{1}{\tau} (\phi (S_{a_h}^n - s_{a}^n), w_h) 
				\\
				+ C \left(\tau + h + \vertiii{e_{p_h}^{n+1}} + \Vert S_{a_h}^{n} - s_a^{n} \Vert + \Vert S_{v_h}^{n} - s_v^{n} \Vert\right) \vertiii{w_h} + \sigma_a(w_h).
				\label{eq:sa_aux2}
			\end{multline}
			Using \cref{eq:error_split} gives the result.
		\end{proof}
		
		\begin{lemma}[Error estimates for the aqueous saturation]
			We have that, for all $0 \leq n \leq N-1$,
			\begin{multline}
				\norm[1]{S_{a_h}^{n+1} - s_a^{n+1}}^2 
				+ \tilde{C} \tau \vertiii{e_{a_h}^{n+1}}^2
				\leq (1 + C\tau)\norm[1]{S_{a_h}^{n} - s_a^{n}}^2
				+ C \tau \norm[1]{S_{v_h}^{n} - s_v^{n}}^2 + C\tau(\tau^2 + h^2),
				\label{eq:error_estimate_aqueous_saturation}
			\end{multline}
			where $C, \tilde{C} > 0$ are independent of $h$ and $\tau$.
			\label{lem:error_estimate_aqueous_saturation}
		\end{lemma}
		
		\begin{proof}
			Let $w_h = \tau e_{a_h}^{n+1}$ in \cref{eq:error_eq_aqueous_saturation}:
			\begin{multline}
				\phi (S_{a_h}^{n+1} - s_{a}^{n+1}, e_{a_h}^{n+1}) 
				+ \tau b_{a}^{n}(e_{a_h}^{n+1}, e_{a_h}^{n+1}) 
				= \phi (S_{a_h}^n - s_{a}^n, e_{a_h}^{n+1}) + \tau b_{a}^{n}(e_{a_\pi}^{n+1}, e_{a_h}^{n+1})
				\\
				+ C \tau \left(\tau + h + \vertiii{e_{p_h}^{n+1}} + \Vert S_{a_h}^{n} - s_a^{n} \Vert\right) \vertiii{e_{a_h}^{n+1}} + \tau \sigma_a(e_{a_h}^{n+1}).
			\end{multline}
			Using the coercivity of $b_a^n$ \cref{eq:coercivity_aqueous_saturation} on the second term of the left-hand side, and the boundedness of $b_a^n$ \cref{eq:boundedness_continuous_saturation} on the first term of the right-hand side, we obtain
			\begin{multline}
				\phi (S_{a_h}^{n+1} - s_{a}^{n+1}, e_{a_h}^{n+1}) 
				+ C_{\alpha,s_a} \tau \vertiii{e_{a_h}^{n+1}}^2
				\leq \phi (S_{a_h}^n - s_{a}^n, e_{a_h}^{n+1}) + C_{B\ast,s_a} \tau \vertiii{e_{a_\pi}^{n+1}}_\ast \vertiii{e_{a_h}^{n+1}}
				\\
				+ C \tau \left(\tau + h + \vertiii{e_{p_h}^{n+1}} + \Vert S_{a_h}^{n} - s_a^{n} \Vert\right) \vertiii{e_{a_h}^{n+1}} + \tau \sigma_a(e_{a_h}^{n+1}).
				\label{eq:sa_est1}
			\end{multline}
			Note that $(S_{a_h}^{n+1} - s_{a}^{n+1}, e_{a_h}^{n+1}) = \norm[1]{S_{a_h}^{n+1} - s_a^{n+1}}^2 + (S_{a_h}^{n+1} - s_{a}^{n+1}, e_{a_\pi}^{n+1})$ and $(S_{a_h}^n - s_{a}^n, e_{a_h}^{n+1}) = (S_{a_h}^n - s_{a}^n, S_{a_h}^{n+1} - s_{a}^{n+1}) + (S_{a_h}^n - s_{a}^n, e_{a_\pi}^{n+1})$. Moreover,
			\begin{equation*}
				(S_{a_h}^n - s_{a}^n, S_{a_h}^{n+1} - s_{a}^{n+1}) = \frac{1}{2}\norm[1]{S_{a_h}^{n+1} - s_a^{n+1}}^2 + \frac{1}{2}\norm[1]{S_{a_h}^{n} - s_a^{n}}^2 
				- \frac{1}{2}\norm[1]{S_{a_h}^{n+1} - s_a^{n+1} - S_{a_h}^n + s_a^n}^2,
			\end{equation*}
			where we have used that $ab = \frac{1}{2}a^2 + \frac{1}{2}b^2 - \frac{1}{2}(a-b)^2$. Thus,
			\begin{multline}
				\frac{\phi}{2}\norm[1]{S_{a_h}^{n+1} - s_a^{n+1}}^2 
				+ C_{\alpha,s_a} \tau \vertiii{e_{a_h}^{n+1}}^2
				\leq \frac{\phi}{2}\norm[1]{S_{a_h}^{n} - s_a^{n}}^2 + \phi(S_{a_h}^n - S_{a_h}^{n+1}, e_{a_\pi}^{n+1}) + \phi(s_{a}^{n+1} - s_{a}^{n}, e_{a_\pi}^{n+1})
				\\
				+ C_{B\ast,s_a} \tau \vertiii{e_{a_\pi}^{n+1}}_\ast \vertiii{e_{a_h}^{n+1}}
				+ C \tau \left(\tau + h + \vertiii{e_{p_h}^{n+1}} + \Vert S_{a_h}^{n} - s_a^{n} \Vert\right) \vertiii{e_{a_h}^{n+1}} + \tau \sigma_a(e_{a_h}^{n+1}).
			\end{multline}
			The second term on the right hand side is zero due to \cref{eq:L2_property}. Moreover, the third term on the right hand side is bounded above by $C\tau \Vert e_{a_\pi}^{n+1} \Vert \leq C \tau h^2$. Thus,
			\begin{multline}
				\frac{\phi}{2}\norm[1]{S_{a_h}^{n+1} - s_a^{n+1}}^2 
				+ C_{\alpha,s_a} \tau \vertiii{e_{a_h}^{n+1}}^2
				\leq \frac{\phi}{2}\norm[1]{S_{a_h}^{n} - s_a^{n}}^2
				+ C_{B\ast,s_a} \tau \vertiii{e_{a_\pi}^{n+1}}_\ast \vertiii{e_{a_h}^{n+1}}
				\\
				+ C \tau \left(\tau + h + \vertiii{e_{p_h}^{n+1}} + \Vert S_{a_h}^{n} - s_a^{n} \Vert\right) \vertiii{e_{a_h}^{n+1}} + \tau \sigma_a(e_{a_h}^{n+1}) + C\tau h^2.
			\end{multline}
			Note that, using the Cauchy--Schwarz inequality, we have $\tau \sigma_a(e_{a_h}^{n+1}) \leq \norm[0]{\beta_a^{n+1}}\,\,\norm[0]{e_{a_h}^{n+1}}$. Moreover,
			\begin{align*}
				\norm[1]{\beta_a^{n+1}}^2 &= \int_{\Omega}\del{\phi \int_{t_n}^{t_{n+1}}(t-t_n)\partial_{tt}s_a\dif t}^2 
				\leq C \int_{\Omega}\del{ \int_{t_n}^{t_{n+1}}(t-t_n)\partial_{tt}s_a\dif t}^2 
				\\
				& \leq C\int_{\Omega}\del{\int_{t_n}^{t_{n+1}}(t-t_n)^2\dif t} \del{\int_{t_n}^{t_{n+1}}(\partial_{tt}s_a)^2\dif t} 
				\\
				& \leq C \tau^3 \int_{\Omega} \int_{t_n}^{t_{n+1}}(\partial_{tt}s_a)^2\dif t 
				\\
				& \leq C \tau^4 \max_{t \in [t_n, t_{n+1}]} \int_\Omega (\partial_{tt}s_a)^2
				\\
				& = C  \tau^4 \max_{t \in [t_n, t_{n+1}]} \norm[1]{\partial_{tt}s_a}^2
				\\
				& \leq C\tau^4,
			\end{align*}
			where we have used that $s_a \in C^2(0,T;L^2(\Omega))$. Using this, Young's inequality on the right hand side, and the estimates \cref{eq:StabNorm_L2projerrors}, we obtain
			\begin{multline}
				\frac{\phi}{2}\norm[1]{S_{a_h}^{n+1} - s_a^{n+1}}^2 
				+ C \tau \vertiii{e_{a_h}^{n+1}}^2
				\leq \frac{\phi}{2}\norm[1]{S_{a_h}^{n} - s_a^{n}}^2
				+ C \tau \left(\tau^2 + h^2 + \vertiii{e_{p_h}^{n+1}}^2 + \Vert S_{a_h}^{n} - s_a^{n} \Vert^2 \right).
				\label{eq:sa_est2}
			\end{multline}
			Using the liquid pressure error estimates \cref{eq:error_estimate_pressure} gives the result.
		\end{proof}
		
		%---------------------------------------------
		\subsection{Vapor saturation}
		%---------------------------------------------
		
		The following two lemma state error estimates for the error $e_{v_h}^{n}$. Proofs are similar to those of 
		Lemma~\ref{lem:error_eq_aqueous_saturation} and Lemma~\ref{lem:error_estimate_aqueous_saturation} and thus, are skipped for brevity.
		
		\begin{lemma}[Error equation for the vapor saturation]
			We have that, for all $w_h \in X_{h,\Gamma_\D^{s_v}}$, and all $0 \leq n \leq N-1$, there exists a constant $C>0$ independent of $h$ and $\tau$ such that
			\begin{multline}
				\dfrac{1}{\tau} (\phi (S_{v_h}^{n+1} - s_{v}^{n+1}), w_h) 
				+ b_{v}^{n}(e_{v_h}^{n+1}, w_h) 
				= \dfrac{1}{\tau} (\phi (S_{v_h}^n - s_{v}^n), w_h) + b_{v}^{n}(e_{v_\pi}^{n+1}, w_h)
				\\
				+ C \left(\tau + h + \vertiii{e_{p_h}^{n+1}} + \Vert S_{a_h}^{n} - s_a^{n} \Vert + \Vert S_{v_h}^{n} - s_v^{n} \Vert\right) \vertiii{w_h} + \sigma_v(w_h),
				\label{eq:error_eq_vapor_saturation}
			\end{multline}
			where
			\begin{equation}
				\sigma_v(w_h) = \frac{1}{\tau} (\beta_v^{n+1},w_h), \quad \beta_v^{n+1} = \phi\int_{t_n}^{t_{n+1}}(t-t_n)\partial_{tt}s_v\dif t.
			\end{equation}
		\end{lemma}
		
		\begin{lemma}[Error estimates for the vapor saturation]
			We have that, for all $0 \leq n \leq N-1$
			\begin{multline}
				\norm[1]{S_{v_h}^{n+1} - s_v^{n+1}}^2 
				+ \tilde{C} \tau \vertiii{e_{v_h}^{n+1}}^2
				\leq (1 + C\tau)\norm[1]{S_{v_h}^{n} - s_v^{n}}^2
				+ C \tau \norm[1]{S_{a_h}^{n} - s_a^{n}}^2 + C\tau(\tau^2 + h^2),
				\label{eq:error_estimate_vapor_saturation}
			\end{multline}
			where $C, \tilde{C} > 0$ are independent of $h$ and $\tau$.
			\label{lem:error_estimate_vapor_saturation}
		\end{lemma}
		
		%---------------------------------------------
		\subsection{Final estimates}
		%---------------------------------------------
		
		In this section we combine the error estimates \cref{eq:error_estimate_pressure}, \cref{eq:error_estimate_aqueous_saturation} and \cref{eq:error_estimate_vapor_saturation}, and use induction to give the final error bounds.
		We first denote the errors made with the starting values by $\mathcal{E}(t_{0})$.
		\[
		\mathcal{E}(t_{0}) = \Vert S_{a_h}^{0}-s_a^0\Vert^2 + \Vert S_{v_h}^{0} - s_v^0\Vert^2.
		\]
		
		\begin{theorem}
			There exists a constant $C$ independent of $h$ and $\tau$ such that the following error estimates hold
			\begin{align}
				\Vert S_{a_h}^{N} - s_a^N \Vert^2 
				+\Vert S_{v_h}^{N} - s_v^N \Vert^2 
				+ C \tau \left( \vertiii{e_{p_h}^N}^2 + \vertiii{e_{a_h}^N}^2 + \vertiii{e_{v_h}^N}^2\right) 
				\leq e^{CT}\left(\mathcal{E}(t_0) + C (\tau^2 + h^2)\right).
				\label{eq:final_error_estimates}
			\end{align}
		\end{theorem}
		
		\begin{proof}
			Let $A_{n+1} = \Vert S_{a_h}^{n+1} - s_a^{n+1} \Vert^2 
			+\Vert S_{v_h}^{n+1} - s_v^{n+1} \Vert^2$, $B_{n+1} = C \tau \left( \vertiii{e_{p_h}^{n+1}}^2 + \vertiii{e_{a_h}^{n+1}}^2 + \vertiii{e_{v_h}^{n+1}}^2\right)$, and $D = C\tau(\tau^2 + h^2)$. Then, adding up all three estimates \cref{eq:error_estimate_pressure}, \cref{eq:error_estimate_aqueous_saturation} and \cref{eq:error_estimate_vapor_saturation}, we obtain
			\[
			A_{n+1} + B_{n+1} \leq (1+C\tau)A_n + D.
			\]
			Applying induction, we have that, for any $1 \leq n \leq N$:
			\[
			A_{n} + B_{n} \leq (1+C\tau)^{n}A_0 + D \sum_{k = 0}^{n-1}(1+C\tau)^k.
			\]
			We apply this with $n = N$. Since $(1+C\tau)^k \leq (1+C\tau)^{N} \leq e^{CN\tau} = e^{CT}$, we have
			\[
			A_{N} + B_{N} \leq e^{CT}A_0 + D \sum_{k = 0}^{n-1}e^{CT} \leq e^{CT}A_0 + (N-1)D e^{CT} \leq e^{CT}(A_0 + C(\tau^2 + h^2)),
			\]
			which concludes the proof.
		\end{proof}
		
		\begin{corollary}
			Assume that the initial solutions $S_{a_h}^0, S_{v_h}^0$ satisfy \eqref{eq:L2projinit}. Then we have
			\begin{align}
				\Vert S_{a_h}^{N} - s_a^N \Vert^2 
				+\Vert S_{v_h}^{N} - s_v^N \Vert^2 
				+ C \tau \left( \vertiii{e_{p_h}^N}^2 + \vertiii{e_{a_h}^N}^2 + \vertiii{e_{v_h}^N}^2\right) 
				\leq Ce^{CT}\left(\tau^2 + h^2\right)
				\label{eq:errors_time_n}
			\end{align}
		\end{corollary}
		\begin{proof}
			This follows from \eqref{eq:final_error_estimates}.
		\end{proof}
		
		%	\Rd
		%	\begin{corollary}
			%		The errors at the final time $t_n = T_F$ satisfy the following bound
			%		\begin{align}
				%			\Vert P_h^{N} - p_\ell^N \Vert^2 + \tau \vertiii{e_{p_h}^N}^2
				%			+\Vert S_{a_h}^{N} - s_a^N \Vert^2 + \tau \vertiii{e_{a_h}^N}^2
				%			+\Vert S_{v_h}^{N} - s_v^N \Vert^2 + \tau \vertiii{e_{v_h}^N}^2
				%			\leq C (\tau^2 + h^2). 
				%		\end{align}
			%	\end{corollary}
		%	
		%	\begin{proof}
			%		The result is obtained by setting $n=N$ in \cref{eq:errors_time_n} and using that $N\tau = T$.
			%	\end{proof}
		
		\begin{remark}
			Note that, as expected, the convergence rates in space are suboptimal. As we show in \cref{sec:numerical_results}, by setting $\tau = h^2$ we can recover second order convergence. In order to obtain optimal rates of convergence in space, a duality argument is needed.
		\end{remark}
		
		%=============================================
		\section{Numerical results}
		%=============================================
		\label{sec:numerical_results}
		
		For the numerical results we consider manufactured solutions under different scenarios. The solution of the problem is given by
		\begin{subequations}
			\begin{align}
				p_\ell(t,x,y) & = 2 + xy^2 + x^2\sin(t+y),
				\\
				s_a(t,x,y) &= \frac{1+2x^2y^2+\cos(t+x)}{8},
				\\
				s_v(t,x,y) &= \frac{3 - \cos(t+x)}{8}.
			\end{align}
		\end{subequations}
		The computational domain is taken as $\Omega = [0,1] \times [0,1]$, and the final time of the problem is $T = 1$. The porosity $\phi$ is taken to be constant equal to 0.2, while the absolute permeability $\kappa$ is taken to be constant equal to 1. The phase viscosities are set as follows:
		\begin{equation}
			\mu_\ell = 0.75, \quad \mu_v = 0.25, \quad \mu_a = 0.5.
		\end{equation}
		The phase relative permeabilities and the capillary pressures are defined as follows~\cite{Bentsen:1976,Chen:2006}:
		\begin{subequations}
			\begin{align}
				k_{r\ell} = s_\ell(s_\ell+s_a)(1-s_a), \quad k_{rv} = s_v^2, \quad k_{ra} = s_a^2,
				\\
				p_{c,v} = \frac{3.9}{\text{ln}(0.01)}\text{ln}(1.01 - s_v), \quad p_{c,a} = \frac{6.3}{\text{ln}(0.01)}\text{ln}(s_a + 0.01).
			\end{align}
		\end{subequations}
		We consider Dirichlet boundary conditions on all the boundaries of the domain. The source terms $q_\ell$, $q_v$ and $q_a$ are computed according to the manufactured solutions and other parameters of the problem. 
		
		%--------------------------------------
		\subsection{Constant densities}
		%--------------------------------------
		\label{subsec:constant_densities}
		
		First, we consider the case in which the phase densities are constant and taken as
		\begin{equation}
			\rho_\ell = 3, \quad \rho_v = 1, \quad \rho_a = 5.
			\label{eq:densities_test}
		\end{equation}
		For this test case, gravity is not considered. We take $\theta_{p_\ell} = \theta_{s_a} = \theta_{s_v} = 1$ and $\alpha_{p_\ell,e} = \alpha_{s_a,e} = \alpha_{s_v,e} = 1$, on all the edges of the mesh. The simulation is performed on six uniform meshes with an initial mesh size of $h = 0.5$. We compute the $L^2$-errors at the final time. In \cref{tab:ConvRates_tau_equals_h_constant_rho} and \cref{tab:ConvRates_tau_equals_h_square_constant_rho} we show the results with the time step $\tau$ taken equal to $h$ and to $h^2$, respectively. We observe that, as expected from the results in \cref{sec:error_analysis}, when $\tau = h$, the scheme is first order. Moreover, we can recover second order when taking $\tau = h^2$.
		
		\begin{center}
			\begin{table}[tbp]
				\centering
				\small
				\caption{Rates of convergence for test case in \cref{subsec:constant_densities}, with $\tau = h$.}
				\begin{tabular}{cccccccc}
					\toprule
					\multirow{2}{*}{}
					& & \multicolumn{2}{c}{$p_\ell$}  & \multicolumn{2}{c}{$s_a$}  & \multicolumn{2}{c}{$s_v$}\\
					$h$ & {DOFs}
					& {$L^2(\Omega)$-error} & {Rate}
					& {$L^2(\Omega)$-error} & {Rate} & {$L^2(\Omega)$-error} & {Rate}\\
					\midrule
					{0.25} & {64} & {3.18e-2} & {-} & {7.41e-3} & {-} & {5.84e-2} & {-} \\
					{0.125} & {256}  & {1.14e-2} & {1.48} & {4.67e-3} & {0.66} & {9.64e-3} & {2.60} \\
					{0.0625} & {1,024} & {2.78e-3} & {2.04} & {2.27e-3} & {1.04} & {4.77e-3} & {1.02} \\
					{0.03125} & {4,096} & {9.22e-4} & {1.59} & {1.18e-3} & {0.94} & {2.15e-3} & {1.15} \\
					{0.015625} & {16,384} & {3.41e-4} & {1.44} & {6.01e-4} & {0.97} & {1.08e-3} & {1.01} \\
					\bottomrule
				\end{tabular}
				\label{tab:ConvRates_tau_equals_h_constant_rho}
			\end{table}
		\end{center}
		
		\begin{center}
			\begin{table}[tbp]
				\centering
				\small
				\caption{Rates of convergence for test case in \cref{subsec:constant_densities}, with $\tau = h^2$.}
				\begin{tabular}{cccccccc}
					\toprule
					\multirow{2}{*}{}
					& & \multicolumn{2}{c}{$p_\ell$}  & \multicolumn{2}{c}{$s_a$}  & \multicolumn{2}{c}{$s_v$}\\
					$h$ & {DOFs}
					& {$L^2(\Omega)$-error} & {Rate}
					& {$L^2(\Omega)$-error} & {Rate} & {$L^2(\Omega)$-error} & {Rate}\\
					\midrule
					{0.5} & {16} & {1.36e-1} & {-} & {6.48e-3} & {-} & {5.11e-2} & {-} \\
					{0.25} & {64} & {3.40e-2} & {2.00} & {1.51e-3} & {2.10} & {3.37e-3} & {3.92} \\
					{0.125} & {256}  & {8.43e-3} & {2.01} & {3.74e-4} & {2.01} & {6.95e-4} & {2.28} \\
					{0.0625} & {1,024} & {2.11e-3} & {2.00} & {9.35e-5} & {2.00} & {1.85e-4} & {1.91} \\
					{0.03125} & {4,096} & {5.32e-4} & {1.99} & {2.32e-5} & {2.01} & {5.07e-5} & {1.87} \\
					\bottomrule
				\end{tabular}
				\label{tab:ConvRates_tau_equals_h_square_constant_rho}
			\end{table}
		\end{center}

		\subsection{Gravity}
		%--------------------------------------
		\label{subsec:gravity}
		
		Finally, we consider the effect of gravity. We take $\boldsymbol{g} = [0\,\, -0.1]^T$, and the phase densities as taken as in \cref{eq:densities_test}. The simulation is performed on six uniform meshes with an initial mesh size of $h = 0.5$. We compute the $L^2$-errors at the final time. In \cref{tab:ConvRates_tau_equals_h_gravity,tab:ConvRates_tau_equals_h_square_gravity} we show the results with the time step $\tau$ taken equal to $h$ and to $h^2$, respectively. We observe that, as expected from the results in \cref{sec:error_analysis}, when $\tau = h$, the scheme is first order. Moreover, we can recover second order when taking $\tau = h^2$.
		
		\begin{center}
			\begin{table}[tbp]
				\centering
				\small
				\caption{Rates of convergence for test case in \cref{subsec:gravity}, with $\tau = h$.}
				\begin{tabular}{cccccccc}
					\toprule
					\multirow{2}{*}{}
					& & \multicolumn{2}{c}{$p_\ell$}  & \multicolumn{2}{c}{$s_a$}  & \multicolumn{2}{c}{$s_v$}\\
					$h$ & {DOFs}
					& {$L^2(\Omega)$-error} & {Rate}
					& {$L^2(\Omega)$-error} & {Rate} & {$L^2(\Omega)$-error} & {Rate}\\
					\midrule
					{0.25} & {64} & {3.20e-2} & {-} & {8.10e-3} & {-} & {6.05e-2} & {-} \\
					{0.125} & {256}  & {1.20e-2} & {1.42} & {5.06e-3} & {0.68} & {1.11e-2} & {2.45} \\
					{0.0625} & {1,024} & {2.78e-3} & {2.11} & {2.42e-3} & {1.06} & {5.03e-3} & {1.14} \\
					{0.03125} & {4,096} & {9.78e-4} & {1.51} & {1.27e-3} & {0.93} & {2.08e-3} & {1.27} \\
					{0.015625} & {16,384} & {3.66e-4} & {1.42} & {6.47e-4} & {0.97} & {1.04e-3} & {1.00} \\
					\bottomrule
				\end{tabular}
				\label{tab:ConvRates_tau_equals_h_gravity}
			\end{table}
		\end{center}
		
		\begin{center}
			\begin{table}[tbp]
				\centering
				\small
				\caption{Rates of convergence for test case in \cref{subsec:gravity}, with $\tau = h^2$.}
				\begin{tabular}{cccccccc}
					\toprule
					\multirow{2}{*}{}
					& & \multicolumn{2}{c}{$p_\ell$}  & \multicolumn{2}{c}{$s_a$}  & \multicolumn{2}{c}{$s_v$}\\
					$h$ & {DOFs}
					& {$L^2(\Omega)$-error} & {Rate}
					& {$L^2(\Omega)$-error} & {Rate} & {$L^2(\Omega)$-error} & {Rate}\\
					\midrule
					{0.5} & {16} & {1.36e-1} & {-} & {6.53e-3} & {-} & {5.50e-2} & {-} \\
					{0.25} & {64} & {3.43e-2} & {1.99} & {1.56e-3} & {2.07} & {3.72e-3} & {3.89} \\
					{0.125} & {256}  & {8.47e-3} & {2.02} & {3.79e-4} & {2.04} & {6.55-4} & {2.51} \\
					{0.0625} & {1,024} & {2.13e-3} & {1.99} & {9.51e-5} & {1.99} & {1.81e-4} & {1.86} \\
					{0.03125} & {4,096} & {5.35e-4} & {1.99} & {2.37e-5} & {2.00} & {5.03e-5} & {1.85} \\
					\bottomrule
				\end{tabular}
				\label{tab:ConvRates_tau_equals_h_square_gravity}
			\end{table}
		\end{center}
		
		%	\subsubsection{Second order}
		%	\label{subsubsec:bdf2}
		%	
		%	We can apply a second order backward differentiation formula (BDF2) to approximate the time derivative, and a second order extrapolation to approximate the coefficients and linearize the problem. Below we show the rates of convergence with the same parameters as 
		%	in \cref{subsec:gravity}, but we take $\mu_\ell = 2$, $\mu_a = 1$ and $\mu_v = 10$.
		%	
		%	\begin{center}
			%		\begin{table}[tbp]
				%			\centering
				%			\small
				%			\caption{Rates of convergence for test case in \cref{subsubsec:bdf2}, with $\tau = h/2$.}
				%			\begin{tabular}{cccccccc}
					%				\toprule
					%				\multirow{2}{*}{}
					%				& & \multicolumn{2}{c}{$p_\ell$}  & \multicolumn{2}{c}{$s_a$}  & \multicolumn{2}{c}{$s_v$}\\
					%				$h$ & {DOFs}
					%				& {$L^2(\Omega)$-error} & {Rate}
					%				& {$L^2(\Omega)$-error} & {Rate} & {$L^2(\Omega)$-error} & {Rate}\\
					%				\midrule
					%				{0.5} & {16} & {1.56e-1} & {-} & {1.25e-2} & {-} & {9.42e-3} & {-} \\
					%				{0.25} & {64} & {3.78e-2} & {2.05} & {1.93e-3} & {2.70} & {3.39e-3} & {1.47} \\
					%				{0.125} & {256}  & {9.43e-3} & {2.00} & {4.22e-4} & {2.19} & {1.18e-3} & {1.52} \\
					%				{0.0625} & {1,024} & {2.39e-3} & {1.98} & {9.38e-5} & {2.17} & {3.73e-4} & {1.66} \\
					%				{0.03125} & {4,096} & {6.06e-4} & {1.98} & {2.27e-5} & {2.05} & {1.08e-4} & {1.79} \\
					%				{0.015625} & {16,384} & {1.53e-4} & {1.99} & {5.62e-6} & {2.01} & {2.92e-5} & {1.89} \\
					%				\bottomrule
					%			\end{tabular}
				%			\label{tab:ConvRates_tau_equals_h_gravity_second_order}
				%		\end{table}
			%	\end{center}
		
		%=============================================
		\section{Conclusions}
		%=============================================
		\label{sec:conclusions}
		We presented and analyzed a first order discontinuous Galerkin method for the incompressible three-phase flow problem in porous media. Our method does not require a subiteration scheme which makes it computationally cheaper. We obtained a priori error estimates by assuming Lipschitz continuity of the coefficients. The numerical test cases show, under different scenarios, that our scheme is first order convergent. For future work, we would like to extend the numerical analysis to variable density flow and to a second order scheme by using a BDF2 time stepping. Moreover, we plan to extend this scheme to the black oil problem where mass transfer can occur between the liquid and vapor phases and study the performance of such methods on setups that includes wells or viscous fingering effects ~\cite{bangerth2006optimization,li2016numerical}.
		
		%========
		%\section{Appendix}
		%========

		%---------------------------------------------
		\bibliographystyle{siamplain}
		\bibliography{references}
		%---------------------------------------------
	\end{document}